\documentclass[a4paper,10pt]{amsart}

\usepackage{amsthm,amsmath,amssymb,amsfonts}
\usepackage[utf8]{inputenc}
\usepackage{verbatim}
\usepackage{marvosym}
\usepackage{stackrel}
\usepackage{graphicx}
\usepackage[svgnames]{xcolor}
\definecolor{blue(munsell)}{rgb}{0.0, 0.5, 0.69}
\usepackage{lipsum}
\usepackage{svg}
\usepackage{caption}
\usepackage{enumitem}
\usepackage{epigraph}
\usepackage{ebproof}
\usepackage[cal=boondoxo]{mathalfa}
\usepackage[all,cmtip]{xy}
\usepackage{mathtools}
\usepackage{color}
\usepackage{mathrsfs}
\usepackage{subfig}
\usepackage{framed}
\usepackage{hyperref}
\hypersetup{hypertexnames = false, bookmarksdepth = 2, bookmarksopen = true, colorlinks, linkcolor = black, citecolor = blue(munsell), urlcolor = blue(munsell), pdfstartview={XYZ null null 1}}
\usepackage{cleveref}
\usepackage{tikz}
\usepackage{pxfonts}
\usepackage{bbold}
 
\usepackage[colorinlistoftodos]{todonotes}
\usepackage{xparse}
\DeclareDocumentCommand\issue{g}{\todo[size=\footnotesize,color = green!40]{Issue\IfNoValueF{#1}{: #1}}}
\DeclareDocumentCommand\tobedone{g}{\todo[size=\footnotesize,color = yellow!50]{To do\IfNoValueF{#1}{: #1}}}
\DeclareDocumentCommand\notationissue{g}{\todo[size=\footnotesize,color = red!30]{Notation?\IfNoValueF{#1}{: #1}}}
\DeclareDocumentCommand\doubt{g}{\todo[size=\footnotesize,color = blue!10]{Doubt\IfNoValueF{#1}{: #1}}}
\DeclareDocumentCommand\observation{g}{\todo[size=\footnotesize,color = orange!10]{Observation\IfNoValueF{#1}{: #1}}}

\usepackage[all,cmtip]{xy}

\usetikzlibrary{calc}
\usetikzlibrary{matrix,arrows,arrows.meta}
\usepackage{tikz-cd}

\usepackage{calligra} 

\newtagform{fn}{(}{)\footnotemark}

    {\endMakeFramed}

    {\endMakeFramed}

\makeatletter
\g@addto@macro\bfseries{\boldmath}
\makeatother

\theoremstyle{plain}
\newtheorem{thm}{Theorem}[section]
\newtheorem*{thm*}{Theorem}
\newtheorem{prop}[thm]{Proposition}
\newtheorem{lem}[thm]{Lemma}
\newtheorem{cor}[thm]{Corollary}

\theoremstyle{definition}
\newtheorem{defn}[thm]{Definition}

\theoremstyle{remark}
\newtheorem{rem}[thm]{Remark}
\newtheorem{exa}[thm]{Example}

\newcommand\Ind{\operatorname{Ind}}

\newcommand\catpresl{\kappa-\ensuremath{\mathsf{LP}}}

\newcommand\Mod{\operatorname{Mod}}

\newcommand\Alg{\operatorname{Alg}}

\newcommand\Lan{\operatorname{Lan}}

\newcommand\opp{\operatorname{op}}

\newcommand\id{\operatorname{id}}
\newcommand\Id{\operatorname{Id}}
\newcommand\coeq{\operatorname{Coeq}}

\newcommand\Set{\ensuremath{\mathsf{Set}}}

\newcommand\Ab{\ensuremath{\mathsf{Ab}}}

\newcommand\Lex{\ensuremath{\mathsf{Lex}}}
\newcommand\Rex{\ensuremath{\mathsf{Rex}}}

\newcommand\Pres{\operatorname{Pres}}
\newcommand\Prototopoi{\ensuremath{\mathsf{Prototopoi}}}
\newcommand\Grtopoi{\ensuremath{\mathsf{GrTopoi}}}
\newcommand{\Sites}{\ensuremath{\mathsf{Sites}}}
\newcommand{\RegCard}{\mathrm{RegCard}}
\newcommand\op{\circ}
\newcommand\coop{\mathrm{coop}}
\newcommand\Sh{\operatorname{Sh}}

\newcommand\ca{\mathcal {A}}
\newcommand\cb{\mathcal {B}}
\newcommand\cc{\mathcal {C}}

\newcommand\cg{\mathcal {G}}

\newcommand\cs{\mathcal {S}}

\newcommand\cv{\mathcal {V}}

\newcommand{\AAA}{\ensuremath{\mathsf{A}}}
\newcommand{\BBB}{\ensuremath{\mathsf{B}}}
\newcommand{\CCC}{\ensuremath{\mathsf{C}}}
\newcommand{\DDD}{\ensuremath{\mathsf{D}}}
\newcommand{\III}{\ensuremath{\mathsf{I}}}

\title{Gabriel-Ulmer duality for topoi and its relation with site presentations}
\author{Ivan {Di Liberti}} 
\thanks{The first named author is supported by grants GA17-27844S and MUNI/A/1103/2017}

\author{Julia {Ramos Gonz\'alez}} 
\thanks{The second named author is a Postdoctoral Fellow of the Research Foundation - Flanders (FWO) and acknowledges the support of the Research Foundation - Flanders (FWO) under grant No.G.0D86.16N during the first months of working on this paper.}

\address{
Ivan \textsc{Di Liberti}: \newline
Department of Mathematics and Statistics\newline
Masaryk University, Faculty of Sciences\newline
Kotl\'{a}\v{r}sk\'{a} 2, 611 37 Brno, Czech Republic\newline
\href{mailto:diliberti@math.muni.cz}{\sf diliberti@math.muni.cz}
}

\address{
	Julia \textsc{Ramos Gonz\'alez}: \newline
	Department of Mathematics and Computer Science\newline
	University of Antwerp, Faculty of Sciences\newline
	Middelheimcampus, Middelheimlaan 1, 2020 Antwerp, Belgium\newline
	\href{mailto:julia.ramosgonzalez@uantwerpen.be}{\sf julia.ramosgonzalez@uantwerpen.be}
}

\begin{document}
	\tikzcdset{arrow style=tikz, diagrams={>=to}}
	\begin{abstract} 	
		Let $\kappa$ be a regular cardinal. We study Gabriel-Ulmer duality when one restricts the $2$-category of locally $\kappa$-presentable categories with $\kappa$-accessible right adjoints to its locally full sub-$2$-category of $\kappa$-presentable Grothendieck topoi with geometric $\kappa$-accessible morphisms. In particular, we provide a full understanding of the locally full sub-$2$-category of the $2$-category of $\kappa$-small cocomplete categories with $\kappa$-colimit preserving functors arising as the corresponding $2$-category of presentations via the restriction. We analyse the relation of these presentations of Grothendieck topoi with site presentations and we show that the $2$-category of locally $\kappa$-presentable Grothendieck topoi with geometric $\kappa$-accessible morphisms is a reflective sub-bicategory of the full sub-$2$-category of the $2$-category of sites with morphisms of sites genearated by the weakly $\kappa$-ary sites in the sense of \cite{shulman12}.
	\end{abstract}
	\maketitle

	\tableofcontents
	\section{Introduction}
	The notion of topos was introduced by Grothendieck in the 1960s as a category of sheaves on a site, one of its most well-known incarnations. Since then, Grothendieck topoi have become an essential object of study in many areas of modern mathematics, ranging from algebraic geometry to intuistionistic logic.  
	
	It is well-known that different sites can give rise to the same Grothendieck topos and that given a Grothendieck topos there is no canonical choice of small site to present it, as it was already pointed out in \cite{SGA4-1}. This flexibility of choice makes of topos theory an extremely powerful mathematical tool. Caramello's program ``Toposes as bridges'' \cite{caramello18} is an illustrative example, where Grothendieck topoi presented by sites of different mathematical nature are a key point to connect various a priori unrelated mathematical theories. This flexibility comes however at a price for the daily practice of the working topos theorist. Solving problems naturally living in the $2$-category $\Grtopoi$ of Grothendieck topoi by choosing concrete sheaf presentations allows to work in the more tangible $2$-category $\Sites$ of small sites, but then every construction has to be checked to be independent of the underlying sites chosen. 
	
	Ideally one would like to have a way of presenting Grothendieck topoi by nice small categories functorially, avoiding having to check independence of presentations. One possible approach is zooming out and working with a bigger family of categories for which more rigidity in the choice of presentations (of whichever nature) is available. A classical candidate is the class of locally presentable categories: it is well-known that Grothendieck topoi are indeed locally presentable categories (see for example \cite[Prop 3.4.16]{borceux94-HCA3}) and Gabriel-Ulmer duality \cite[Kor 7.11]{gabrielulmer71} states that locally $\kappa$-presentable categories are completely determined by their small subcategory of $\kappa$-presentable objects, providing a natural presentation of locally $\kappa$-presentable categories in terms of $\kappa$-cocomplete small categories.
	
	Given a regular cardinal $\kappa$, the main goal of this paper is to restrict the $2$-categorical generalization of the classical Gabriel-Ulmer duality to a duality for locally $\kappa$-presentable Grothendieck topoi with $\kappa$-accessible geometric functors in order to understand which $\kappa$-cocomplete categories appear as their natural presentations, and eventually interpret those in terms of classical site presentations.
	
	The structure of the article is as follows. 
	
	In \S\ref{GUduality} we review the $2$-categorical version of Gabriel-Ulmer duality from \cite[Th. 3.1]{centazzovitale02} applied to locally $\kappa$-presentable categories. 
	
	In \S\ref{sectioncharacterization} we focus on the restriction of Gabriel-Ulmer duality to Grothendieck topoi at the level of objects (or $0$-cells), namely we characterize the locally $\kappa$-presentable categories which are Grothendieck topoi in terms of their subcategories of $\kappa$-presentable objects. Historically, this question was initially posed for the case of locally finitely presentable categories in \cite[footnote p. $106$]{SGA4-1} (that is $\kappa = \aleph_0$). The answer appeared for the first time to our knowledge in \cite{carbonipedicchiorosicky01}, providing a characterization of the small categories $\CCC$ with finite colimits whose $\Ind$-completion $\Ind(\CCC)$ (that is the free directed colimit completion of $\CCC$) is a Grothendieck topos. We provide, based on \cite{carbonipedicchiorosicky01}, a generalization of this characterization to higher cardinalities:
	\begin{thm}[{\Cref{characterizationkappa}}]\label{characterizationgeneralization}
		Let $\kappa$ be regular cardinal and let $\CCC$ be a small category closed under $\kappa$-small colimits. The following are equivalent:
		\begin{enumerate}[label=\rm (\arabic*)]
			\item $\CCC$ is $\kappa$-extensive (\Cref{defextensive}) and pro-exact (\Cref{defproexact});
			\item $\Ind_{\kappa}(\CCC)$ is a Grothendieck topos;
		\end{enumerate}
		where $\Ind_{\kappa}$ denotes the free $\kappa$-directed colimit completion, usually referred to as $\Ind_{\kappa}$-completion. A locally $\kappa$-presentable category $\ca$ is hence a Grothendieck topos if and only if its full category of $\kappa$-presentable objects $\Pres_{\kappa}\ca$ is $\kappa$-extensive and pro-exact.
	\end{thm} 
	We say that a small category $\CCC$ with $\kappa$-small colimits is a \emph{$\kappa$-prototopos} (see \Cref{prototopos}) if it satisfies one (and hence both) of the two equivalent properties of the theorem.
	
	In \S\ref{dualitysection} we focus on the restriction of Gabriel-Ulmer duality at the level of $1$-cells. While the $1$-cells considered in Gabriel-Ulmer duality for $\kappa$-presentable categories are the $\kappa$-accessible right adjoints, classically, the $1$-cells usually considered when working with Grothendieck topoi are the geometric morphisms. Therefore, we restrict the duality at the level of $1$-cells to $\kappa$-accessible right adjoints which are part of a geometric morphism. This allows us to present the desired restriction of the result of Gabriel and Ulmer:
	\begin{thm}[{{\Cref{1}}}]\label{ourduality}
		Let $\kappa$ be a regular cardinal.
		There is a biequivalence of $2$-categories $$\Ind_{\kappa}: \Prototopoi_\kappa^{\coop} \rightleftarrows  \Grtopoi_{\kappa} : \Pres_{\kappa},$$
		where $\Prototopoi_\kappa$ is the $2$-category of $\kappa$-prototopoi with $\kappa$-small colimit preserving functors which are $\Ind_{\kappa}$-flat (\Cref{indkappaflat}).
	\end{thm}
	
	We then analyse in \S\ref{cofinality} up to what point the understanding of the family of $2$-categories $\{\Grtopoi_{\kappa}\}_{\kappa \in \RegCard}$ (where $\RegCard$ denotes the set of small regular cardinals) provided by this topoi-prototopoi duality can help us comprehend the $2$-category $\Grtopoi$ of Grothendieck topoi with geometric morphisms and natural transformations, which is the $2$-category of main interest from a topos theory perspective. 
	
	Finally, in \S\ref{sitesandprototopoi}, we study the relation between the prototopos presentations of Grothendieck topoi and the classical site presentations. More concretely, we analyse the connection between the full sub-$2$-category $\Sites_{\kappa}$ of the $2$-category $\Sites$ of small sites with morphisms of sites generated by the weakly $\kappa$-ary sites in the sense of \cite{shulman12} (see \Cref{defnsize}) and the $2$-category $\Prototopoi_{\kappa}$. This allows us to prove the following result:
	\begin{thm}[{{\Cref{semantic}}}]\label{mainresult}
		The $2$-category $\Grtopoi_{\kappa}$ is a reflective sub-bicategory in the $2$-category $\Sites_{\kappa}^{\coop}$, where the left adjoint to the reflection is provided by taking sheaves.
	\end{thm}
	This result allows us to choose canonical small site presentations while working within $\Grtopoi_{\kappa}$, a possibility which is not available when one works in $\Grtopoi$. 
	
	As an immediate corollary of \Cref{mainresult}, we also obtain the desired relation between $\kappa$-prototopoi and sites defining locally $\kappa$-presentable topoi as follows:
	\begin{thm}[{{\Cref{syntactic}}}]
		The $2$-category $\Prototopoi_{\kappa}$ is a reflective bicategory in the $2$-category $\Sites_{\kappa}$.
	\end{thm}
	
	We would like to point out the relation of our work with that of Shulman in \cite{shulman12}. While in this paper we focus on understanding the free $\kappa$-directed colimit completions that yield Grothendieck topoi, Shulman studies in \cite{shulman12} the free exact completions that Grothendieck topoi, and thus both papers are very much related in spirit. We also find an analog of \Cref{ourduality} in \cite[Thm 9.5]{shulman12}.
	
	In this article we have exclusively worked in a non-enriched setup, with our objects of study being classical Grothendiek topoi (enriched over $\Set$). The analogous results in an enriched setup over suitable categories $\cv$ are currently under investigation. For this purpose the authors find the work in \cite{garnerlack12} of great relevance. Our interest in the analysis of the enriched case, apart from being a desirable generalization of the current results, relies on the fact that the enriched framework could provide a better understanding of the notion of prototopos, where pro-exactness appears as an obscure property. Furthermore, we are particularly interested in the case of enrichments over the category of abelian groups $\Ab$. Grothendieck topoi enriched over $\Ab$ are precisely the Grothendieck categories, as one can easily deduce from Gabriel-Popescu theorem together with the theory of enriched sheaves from \cite{borceuxquinteiro96}. Grothendieck categories are objects of essential relevance in algebraic geometry, specially from the point of view of noncommutative algebraic geometry, where they play the role of models for noncommutative spaces (see for example \cite{artintatevandenbergh90},\cite{artinzhang94},\cite{staffordvandenbergh01}). In this setup, $\aleph_0$-prototopoi correspond to \emph{left abelian} additive categories, as shown in \cite[Satz 2.7]{breitsprecher70} (see also \cite{rump10}). We must point out that this result is actually previous to the above-mentioned characterization of $\aleph_0$-prototopoi (in the non-enriched setup) from \cite{carbonipedicchiorosicky01} and hence is not proved as an enriched version of it. Both the interpretation of this result in key of enrichments and its extension to higher cardinalities are work in progress. 

\section{Gabriel-Ulmer duality} \label{GUduality}
Gabriel-Ulmer duality is a syntax-semantics-type duality that states that locally $\kappa$-presentable categories are uniquely determined by their full subcategories of $\kappa$-presentable objects. This section is a brief exposition of this result.
\begin{thm}[Gabriel-Ulmer duality]  \label{GU} There is a biequivalence of $2$-categories
	\begin{equation}
		\Rex_{\kappa}^{\coop} \rightarrow \catpresl 
	\end{equation}
	where:
	\begin{itemize}
		\item $\Rex_{\kappa}^{\coop}$ is the conjugate-opposite $2$-category of the $2$-category $ \Rex_{\kappa}$ of small categories with $\kappa$-small colimits with $1$-cells given by the $\kappa$-small colimit preserving functors and $2$-cells given by the natural transformations;
		\item $\catpresl$ is the $2$-category of locally $\kappa$-presentable categories with $1$-cells given by $\kappa$-accessible right adjoints and $2$-cells given by natural transformations between them.
	\end{itemize}
\end{thm}

This duality dates back to \cite[Kor 7.11]{gabrielulmer71} (at the level of objects). Since then, the duality has been extended to a $2$-categorical version and different generalizations have appeared in the literature, among those, a generalization to more general limit doctrines can be found in \cite[Th 3.1]{centazzovitale02}, a proof of the duality in the enriched context can be found in \cite{lackpower09}, while in \cite[Th 4.12]{dilibertiloregian18} the duality is analysed from the point of view of formal category theory.

The biequivalence of the theorem is given by the following two quasi-inverse pseudofunctors
\begin{equation}
	\Ind_{\kappa}: \Rex_{\kappa}^{\coop} \rightleftarrows \catpresl :\Pres_{\kappa}. 
\end{equation}
The pseudofunctor $\Ind_{\kappa}$ acts as follows: 
\begin{itemize}
	[leftmargin=+.6in]
	\item [\textbf{$0$-cells}] It maps a $\kappa$-cocomplete small category $\CCC$ to $\Ind_{\kappa}(\CCC)$, that is, its $\Ind_{\kappa}$-completion, also known as \emph{free completion under $\kappa$-filtered colimits} (see \cite[\S3]{adamekborceuxlackrosicky02}). Observe that $\Ind_{\kappa}(\CCC) \cong \Lex_{\kappa}(\CCC^{\op},\Set)$, see for example \cite[\S5.5]{gabrielulmer71} or \cite[Thm 2.4]{adamekborceuxlackrosicky02}, and this is a locally $\kappa$-presentable category \cite[Thm 1.46]{adamekrosicky94}.
	\item[\textbf{$1$-cells}]  It maps a $\kappa$-small colimit preserving functor $f: \CCC \to \DDD$ to the adjunction 
	\[S(f) : \Ind_{\kappa}(\CCC) \rightleftarrows \Ind_{\kappa}(\DDD) : Z(f),\]
	where the right adjoint $Z(f)$ can be described as the restriction to $\Ind_{\kappa}(\DDD)$ of the restriction of scalars functor $f^*: [\DDD^{\op},\Set] \to [\CCC^{\op}, \Set]$, which is easily checked to take values in $\Ind_{\kappa}(\CCC)$. Since the $\Ind_{\kappa}$-completion is reflective in the category of presheaves, limits in the $\Ind_{\kappa}$-completion are computed pointwise, and hence, as $f^*$ preserves all limits, so does $Z(f)$. Consequently, by the adjoint functor theorem, $Z(f)$ has a left adjoint, that we denote by $S(f)$. In addition, as shown in \Cref{Sfstructure} below, we have that the diagram 
	\begin{equation*}\label{defSf}
		\vcenter{\xymatrix{
				\CCC \ar[d]^{ \alpha_{\CCC}}\ar[r]^f  & \DDD \ar[d]_{ \alpha_{\DDD}}\\
				\Ind_k (\CCC) \ar[r]_{S(f)} &\Ind_k (\DDD)  \\
		}}
	\end{equation*}
	commutes and thus $S(f)$  preserves $\kappa$-presentable objects because the corestriction of $\alpha_{\CCC}: \CCC \rightarrow \Pres_{\kappa}(\Ind_{\kappa}(\CCC))$ is an equivalence as a direct consequence of \cite[Thm 1.46]{adamekrosicky94} (analogously for $\DDD$)). Therefore $Z(f)$ is $\kappa$-accessible as a direct consequence of \cite[Prop 2.23]{adamekrosicky94}. Hence, the adjunction $S(f) \dashv Z(f)$, that we will denote by $\Ind_{\kappa}(f)$, is indeed a $1$-cell in $\catpresl$. 
	\item[\textbf{$2$-cells}] It maps a natural transformation $\alpha: f \Rightarrow g: \CCC \rightarrow \DDD$ to 
	\[\alpha^*: Z(g) \Rightarrow Z(f): \Ind_{\kappa}(\DDD) \to \Ind_{\kappa}(\CCC)\] 
	defined by $\alpha^*_F(C) \coloneqq F(\alpha_C) : g^*(F)(C) = F(g(C)) \to F(f(C)) = f^*(F)(C)$ for all $C \in \CCC$ and all $F \in \Ind_{\kappa}(\DDD)$.
\end{itemize}
\begin{rem}\label{Sfstructure}
	We know that the functor $S(f)$ is cocontinous, because it is a left adjoint, and that $\alpha_{\CCC}$ is dense (i.e. $\text{Lan}_{\alpha_{\CCC}}(\alpha_{\CCC}) \cong \text{id}_{\Ind_{\kappa} \CCC}$). Consequently, $S(f)$ must coincide with $\text{Lan}_{\alpha_{\CCC}}(\alpha_{\BBB} f)$, indeed
	\begin{align*}
		S(f)  &= S(f) \, \text{id}_{\Ind_{\kappa} (\CCC)}  \\
		& \cong  S(f) \, \text{Lan}_{\alpha_{\CCC}}(\alpha_{\CCC})  \\
		& \cong  \text{Lan}_{\alpha_{\CCC}}(S(f) \, \alpha_{\CCC} )  \\
		& \cong \text{Lan}_{\alpha_{\CCC}}(\alpha_{\DDD} \, f), \\
	\end{align*}
	where we use the key fact that a cocontinous functor preserves Kan extensions. A formal argument that can be found in \cite[Lem 2.16]{dilibertiloregian18} proves that in this case the right adjoint $Z(f)$ has to coincide with $\text{Lan}_{\alpha_{\DDD}\, f}(\alpha_{\CCC})$.
\end{rem}
The pseudofunctor $\Pres_\kappa$ acts as follows:
\begin{itemize}
	[leftmargin=+.6in]
	\item [\textbf{$0$-cells}]  It maps a locally $\kappa$-presentable $\ca$ to $\Pres_{\kappa}(\ca)$, the full subcategory of $\kappa$-presentable objects. Observe that this category is (essentially) small.
	\item[\textbf{$1$-cells}]  It maps an adjunction $L : \ca \rightleftarrows \cb: R$ with $R$ $\kappa$-accessible to the restriction of $L$ to $\Pres_{\kappa}(\ca)$. It is easy to check that the left adjoint of a $\kappa$-accessible functor preserves $\kappa$-presentable objects and thus this restriction actually lands in $\Pres_{\kappa}(\cb)$. We know that the embeddings $\Pres_{\kappa}(\ca) \hookrightarrow \ca$ and $\Pres_{\kappa}(\cb) \hookrightarrow \cb$ create $\kappa$-small colimits, thus they are preserved by the restriction of $L$, since $L$ is a left adjoint. 
	\item[\textbf{$2$-cells}] It maps a natural transformation $\alpha: (L \dashv R) \Rightarrow (L' \dashv R'): \ca \rightarrow \cb$ between two $\kappa$-accessible right adjoints (that is a natural transformation $R \Rightarrow R': \cb \rightarrow \ca$) to the naturally induced natural transformation between the respective left adjoints $L' \Rightarrow L: \ca \rightarrow \cb$ restricted to the restrictions of $L$ and $L'$ to the subcategories of $\kappa$-presentable objects as explained above. 
\end{itemize}
\begin{rem}\label{RexLex}
	Some readers might be familiar with a different presentation of Gabriel-Ulmer duality involving, instead of $\Rex_{\kappa}$, the $2$-category $\Lex_{\kappa}$ of small categories with $\kappa$-small limits, where $1$-cells are the $\kappa$-small limit preserving functors and $2$-cells are the natural transformations. Both versions of the duality are related as follows:
	
	We have a (pseudo)commutative diagram of biequivalences
	\usetagform{fn}
	\begin{equation}
		\begin{tikzcd}
			\Rex^{\coop}_{\kappa} \ar[bend right=11, ddr, "\op"' description] \ar[bend left=11, rr, "\Ind_{\kappa}" description]  & & \catpresl \ar[bend right=16, ddl, "\Pres^{\op}_{\kappa}" description] \ar[bend left=16, ll, "\Pres_{\kappa}" description]\\
			& & \\
			& \Lex^{\opp}_{\kappa}, \ar[ bend right=16, uur,  "\Lex_{\kappa}" description]  \ar[bend right=16, uul, "\op" description]  & \\
		\end{tikzcd}
	\end{equation}
	\footnotetext{Observe that the passage to the conjugate category $\Rex^{\text{co}}$ is needed just because $ [\CCC^{\op},\DDD^{\op}] \cong [\CCC,\DDD]^{\op}$.}
	where $\op$ denotes the biequivalence provided by taking the opposite category and $\Lex_{\kappa}$ is a short notation for the pseudofunctor $\Lex_{\kappa}(-, \Set)$ In particular, we have that
	\[\text{Ind}_{\kappa} \cong \Lex_{\kappa}((-)^{\op}, \Set),\]
	as we have already pointed out at the level of $0$-cells.
\end{rem}
\begin{rem}
	In \cite{makkaipitts87} the authors observed that the typical paradigm of Stone-like dualities can be used to re-enact Gabriel-Ulmer duality as follows:
	\begin{align*}
		\Pres_{\kappa}(\ca)^{\op} & \cong \Lex_{\kappa}(\mathbb{1}, \Pres_{\kappa}(\ca)^{\op}) \\
		& \cong  \Lex_{\kappa}^{\opp}(\Pres_{\kappa}(\ca)^{\op}, \mathbb{1}) \\
		& \stackrel{\text{GU}}{\cong}  \catpresl(\ca, \Lex_{\kappa}(\mathbb{1}, \Set))\\
		&  \cong  \catpresl(\ca,  \Set),
	\end{align*}
	for all $\ca$ locally $\kappa$-presentable category. This chain of isomorphisms proves that the $2$-functor $\Pres_{\kappa}^{\op}$ is in fact represented in the $2$-category $\catpresl$ by the object $\Set$.
\end{rem}

\begin{rem}\label{synsem}
	Gabriel-Ulmer duality is usually referred to as a duality of syntax-semantics kind. This is easily illustrated from the perspective of cartesian logic, also known as finite limit logic, the study of which was introduced in the works by Freyd \cite{freyd02}, Isbell \cite{isbell72} and Coste \cite{coste76}. A category with finite limits $\CCC$ can be seen as a cartesian theory and a functor $M: \CCC \to \Set$ preserving finite limits can be seen as a model of the theory $\CCC$ in $\Set$. The category $\Lex(\CCC, \Set)$ can be thus identified with the category of models of $\CCC$. From this perspective, Gabriel-Ulmer duality establishes a reconstruction result: the category of models of a cartesian theory (the semantics) fully determines the theory (the syntax). Observe that this same observation applies equally to cartesian logic of higher cardinality.
	Accordingly, throughout the rest of the paper, we will refer to the $2$-categories $\Lex_{\kappa}$ and $\Rex_{\kappa}$ as the syntactic side of the duality, while $\catpresl$ will be referred to as the semantic part.
\end{rem}

\section{Grothendieck topoi among locally presentable categories}\label{sectioncharacterization}
In \cite{carbonipedicchiorosicky01} the authors provide a characterization of Grothendieck topoi among locally finitely presentable categories. The result can be stated as follows:
\begin{thm}[{\cite[Thm 5 + Thm 9]{carbonipedicchiorosicky01}}]\label{characterizationfinite}
	Let $\CCC$ be a small category with finite colimits. The following are equivalent:
	\begin{enumerate}[label=\rm (\arabic*)]
		\item $\CCC$ is extensive and pro-exact.
		\item $\Ind(\CCC)$ is a Grothendieck topos.
	\end{enumerate}
	A locally finitely presentable category $\ca$ is hence a Grothendieck topos if and only if its full subcategory of finitely presentable objects $\Pres_{\aleph_0}\ca$ is extensive and pro-exact.
\end{thm} 
The aim of this section is to generalize \Cref{characterizationfinite} to higher cardinals in order to provide the characterization of Grothendieck topoi among locally $\kappa$-presentable categories as stated in \Cref{characterizationgeneralization} above. This result will provide the restriction of Gabriel-Ulmer duality to Grothendieck topoi at the level of objects. 

In order to prove \Cref{characterizationgeneralization} we follow the finitary version of the arguments as presented in \cite{carbonipedicchiorosicky01}. The reader should notice that, while a higher cardinality version of extensivity needs to be considered in order to generalize the proof, this is not the case for pro-exactness. 

The following characterization of Grothendiek topoi (see for example \cite[\S 4.3]{carbonivitale98}), which is a corollary of Giraud's characterization, will be useful for our purposes.
\begin{thm}[Giraud]\label{Giraud} Let $\cc$ be a category. The following are equivalent:
	\begin{enumerate}[label=\rm (\arabic*)]
		\item $\cc$ is a Grothendieck topos.
		\item  $\cc$ is exact (\Cref{defexact}), infinitary extensive (\Cref{defextensive}) and has a set of generators.
	\end{enumerate}
\end{thm}

\subsection{Extensivity and $\kappa$-extensive syntaxes}\label{extensivity}
In \cite[\S 2]{carbonipedicchiorosicky01} a characterization of those small categories whose $\Ind$-completions are extensive is provided. 
The content of this subsection is a generalization of this characterization to higher cardinals.
\begin{defn}\label{defextensive}
	Let $\kappa$ be a regular cardinal. A category $\CCC$ is called \emph{$\kappa$-extensive} (resp. \emph{infinitary extensive}) if it is closed under $\kappa$-small (resp. small) sums and for each $\kappa$-small (resp. small) family $\{X_i\}$ of objects in $\CCC$, the sum functor between the comma categories
	\[
	\prod_i (\CCC/X_i) \xrightarrow{\coprod} \CCC/ ( \coprod_i X_i) 
	\]
	is an equivalence. If $\kappa = \aleph_0$, we simply say that $\CCC$ is \emph{extensive}. A category $\CCC$ is called \emph{$\kappa$-lextensive} (resp. \emph{infinitary lextensive}) if it is $\kappa$-extensive (resp. infinitary extensive) and has all $\kappa$-small (resp. small) limits. If $\kappa = \aleph_0$, we simply say that $\CCC$ is \emph{lextensive}.
\end{defn}
\begin{rem}\label{infinitaryimpliesfinitary}
	Observe that for two regular cardinals $\alpha, \beta$, such that $\alpha \geq \beta$, $\alpha$-extensive implies $\beta$-extensive. In particular, we have that infinitary extensive implies $\kappa$-extensive for every regular cardinal $\kappa$. 
\end{rem}
The following is a useful characterization of $\kappa$-extensive (resp. infinitary extensive) categories obtained as a direct generalization of the finitary case provided in \cite[Lem 1]{carbonipedicchiorosicky01}.
\begin{lem} \label{extensivecharacterization}
	Let $\kappa$ be a regular cardinal. Let $\CCC$ be a category with $\kappa$-small (resp. small) sums. Then
	\begin{enumerate}[label=\rm (\arabic*)]
		\item Given a family $\{X_i\}_{i\in I}$ of objects of $\CCC$ where $I$ is a $\kappa$-small (resp. small) indexing set, the functor
		\[
		\prod_i (\CCC/X_i) \xrightarrow{\coprod} \CCC/ (\coprod_i X_i)
		\]
		is fully faithful if and only if for all families of objects $\{A_i\}_{i \in I}$ and all families of morphisms $\{f_i: A_i \rightarrow X_i\}_{i\in I}$, the commutative squares
		\[
		\begin{tikzcd}
		A_i \arrow[r,"f_i"] \arrow[d,"i_{A_i}"'] &X_i \arrow[d,"i_{X_i}"]\\
		\coprod_i A_i \arrow[r,"\coprod f_i"] &\coprod_iX_i
		\end{tikzcd}
		\]
		are pullbacks.
		\item Given a $\kappa$-small (resp. small) family $\{X_i\}_{i\in I}$ of objects of $\CCC$, the functor
		\[
		\prod_i (\CCC/X_i) \xrightarrow{\coprod} \CCC/ (\coprod_i X_i)
		\]
		is essentially surjective if and only if for each morphism $t: A \rightarrow \coprod_i X_i$ there exists a $\kappa$-small (resp. small) family of morphisms $\{t_i: A_i \rightarrow X_i\}$ and an isomorphism $f: \coprod_i A_i \rightarrow A$ such that $t \, f = \coprod t_i$.
	\end{enumerate}
\end{lem}
%
\begin{rem}\label{subcategoriesextensive}
	Any subcategory of an infinitary extensive category which is closed under finite limits and arbitrary sums is again infinitary extensive. This follows immediately from \Cref{extensivecharacterization}.
\end{rem}

The following is the higher cardinality version of \cite[Thm 5]{carbonipedicchiorosicky01}.
\begin{lem} \label{extensive}
	Let $\CCC$ be a small $\kappa$-cocomplete category. The following are equivalent:
	\begin{enumerate}[label=\rm (\arabic*)]
		\item $\CCC$ is $\kappa$-extensive;
		\item $\Ind_{\kappa}(\CCC)$ is infinitary extensive;
		\item $\Ind_{\kappa}(\CCC)$ is extensive.
	\end{enumerate}
\end{lem}
\begin{proof}
	(1) $\Rightarrow$ (2): The category of models $\CCC^{\op}$-$\Alg$ for the $\kappa$-ary algebraic theory $\CCC^{\op}$ (this is, the category of functors $\CCC^{\op} \rightarrow \Set$ preserving $\kappa$-small products) is nothing but the topos of sheaves of the $\kappa$-extensive Grothendieck topology on $\CCC$, i.e. the Grothendieck topology on $\CCC$ induced by covers given by $\kappa$-small families $\{f_i:X_i \rightarrow X\}$ such that $\coprod_i f_i: \coprod_i X_i \rightarrow X$ is an isomorphism. For this, see for example \cite[\S2]{carbonipedicchiorosicky01} or \cite[\S B.4]{lurieultracategories} for the finitary case, or \cite[Lem 1.5.15]{low16} for the $\kappa$-ary case. We thus have, as a direct application of Giraud's characterization of Grothendieck topoi (see \Cref{Giraud}) that $\CCC^{\op}$-$\Alg$ is infinitary extensive. In addition, we have that $\Ind_{\kappa}(\CCC) = \Lex_{\kappa}(\CCC^{\op}, \Set)$ is a full subcategory of $\CCC^{\op}$-$\Alg$ which is closed under small limits and arbitrary sums. That limits in $\Ind_{\kappa}(\CCC)$ are computed as in $\CCC^{\op}$-$\Alg$ follows immediately from the fact that $\Ind_{\kappa}(\CCC)$ is reflective in $\CCC^{\op}$-$\Alg$. To show that arbitrary sums in $\Ind_{\kappa}(\CCC)$ are computed as in $\CCC^{\op}$-$\Alg$, recall that $\kappa$-filtered colimits are pointwise in $\Ind_{\kappa}(\CCC)$ as well as in $\CCC^{\op}$-$\Alg$ \cite[Thm 1.52 + Rem]{adamekrosicky94} and that $\kappa$-small coproducts of representables are computed both in $\Ind_{\kappa}(\CCC)$ and in $\CCC^{\op}$-$\Alg$ as in $\CCC$ \cite[Prop 6.13]{kelly05}. One then concludes by observing that, in both $\Ind_{\kappa}(\CCC)$ and $\CCC^{\op}$-$\Alg$, every arbitrary sum can be written as a $\kappa$-filtered colimit of $\kappa$-small coproducts of representables. Hence, as $\CCC^{\op}$-$\Alg$ is infinitary extensive, so is $\Ind_{\kappa}(\CCC)$, as a direct consequence of \Cref{subcategoriesextensive}.
	
	(2) $\Rightarrow$ (3): It follows trivially from the definitions (see \Cref{infinitaryimpliesfinitary}).
	
	(3) $\Rightarrow$ (1): The argument is an analogue of the one used in the proof of \cite[Lem 3]{carbonipedicchiorosicky01} for the $\kappa$-ary case. We write it down for convenience of the reader. We are going to use the characterization of $\kappa$-extensive categories provided by \Cref{extensivecharacterization} above. 
	
	Given two $\kappa$-small families $\{A_i\}_{i \in I}$ and $\{X_i\}_{i\in I}$ of objects in $\CCC$, consider the commutative squares in $\CCC$
	\begin{equation}\label{property1}
	\begin{tikzcd}
	A_i \arrow[r,"f_i"] \arrow[d,"i_{A_i}"'] &X_i \arrow[d,"i_{X_i}"]\\
	\coprod_i A_i \arrow[r,"\coprod f_i"] &\coprod_iX_i.
	\end{tikzcd}
	\end{equation}
	We know that the corestriction $\CCC \rightarrow \Ind_{\kappa}(\CCC): T \mapsto h_T$ of the Yoneda functor preserves $\kappa$-small colimits \cite[Prop 6.13]{kelly05}, hence we have the commutative square
	\[
	\begin{tikzcd}
	&[-20pt] h_{A_i} \arrow[r,"h_{f_i}"] \arrow[d,"i_{h_{A_i}}"'] &h_{X_i} \arrow[d,"i_{h_{X_i}}"]& [-20pt]\\
	h_{\coprod_i A_i} \arrow[equal,r] &\coprod_i h_{A_i} \arrow[r,"\coprod h_{f_i}"] &\coprod_ih_{X_i} \arrow[equal,r]&h_{\coprod_i X_i}
	\end{tikzcd}
	\]
	in $\Ind_{\kappa}(\CCC)$ which is a pullback by hypothesis. As limits in $\Ind_{\kappa}(\CCC)$ are the limits computed in $[\CCC^{\op},\Set]$, and the Yoneda functor preserves and creates limits, we have that the diagram (\ref*{property1}) is also a pullback in $\CCC$.
	
	Consider now a morphism $t:A \rightarrow \coprod_i X_i$ in $\CCC$ with $\{X_i\}_i$ a $\kappa$-small family and take its image $h_t: h_A \rightarrow h_{\coprod_i X_i}= \coprod_i h_{X_i}$ in $\Ind_{\kappa}(\CCC)$. By hypothesis there exist an $\kappa$-small family of morphisms $\{f_i:F_i \rightarrow h_{X_i}\}$ and an isomorphism $f: \coprod_i F_i \rightarrow h_A$ in $\Ind_{\kappa}(\CCC)$, such that $h_t \, f = \coprod_i f_i$. Recall that summands of $\kappa$-presentable objects are $\kappa$-presentable. As $\kappa$-presentable objects in $\Ind_{\kappa}(\CCC)$ are precisely the representables, we have that $F_i = h_{A_i}$ with $A_i \in \CCC$ for all $i$ and thus $\coprod_i F_i  = \coprod_i h_{A_i} = h_{\coprod_i A_i}$, where we use again that $\CCC \rightarrow \Ind_{\kappa}(\CCC)$ preserves $\kappa$-small colimits. As the Yoneda functor is fully faithful, $f_i = h_{s_i}: h_{A_i} \rightarrow h_{X_i}$ for a unique morphism $s_i:A_i \rightarrow X_i$ in $\CCC$ and $f = h_r: h_{\coprod_i A_i} \rightarrow h_A$ for an isomorphism $r: \coprod_i A_i \rightarrow A$ in $\CCC$. Hence, we have that 
	\[
	h_{t\, r}=h_t \, h_r  = h_t \, f = \coprod_i f_i = \coprod_i h_{s_i} = h_{\coprod_i s_i},
	\] 
	which by fully-faithfulness of the Yoneda functor, provides that $t \, r = \coprod_i s_i$, which finishes the argument.
\end{proof}

%
%

\subsection{Regularity and weakly regular syntaxes}\label{regularity} In \cite[\S7]{carbonipedicchiorosicky01} a syntactic characterization of the locally finitely presentable categories that are regular is provided. In this subsection we flesh out the generalization of this result to higher cardinalities.

\begin{defn} A category $\CCC$ is \textit{regular} when:
	\begin{itemize}
		\item[(R1)] it is finitely complete;
		\item[(R2)] every kernel pair admits a coequalizer;
		\item[(R3)] the pullback of a regular epimorphism along any morphism is again a regular epimorphism. Recall that an epimorphism is called \emph{regular} if it is a coequalizer of a pair of morphisms. 
	\end{itemize}
\end{defn}

We provide now the definition of \emph{weakly regular} category. To the best of our knowledge, this definition is due to the authors of \cite{carbonipedicchiorosicky01}.

\begin{defn} A category $\CCC$ is \textit{weakly regular} if every commutative square
	\[
	\begin{tikzcd}
	C \ar[r,"f'"] \ar[d,"g'"'] &D \ar[d,"g"]\\
	A\ar[r,"f"] & B.
	\end{tikzcd}
	\]
	in which $f$ is a regular epimorphism factors through a commutative square
	\[
	\begin{tikzcd}
	C \ar[bend right=40, ddr, "g'"'] \ar[bend left=40, drr, "f'"] \ar[dr, dotted] & & \\
	& \bar{C} \arrow[r, dotted, "\bar{f}"] \arrow[d, dotted, "\bar{g}"'] &D \arrow[d,"g"]\\
	& A\arrow[r,"f"] & B.
	\end{tikzcd}
	\]
	where $\bar{f}$ is a regular epimorphism.
\end{defn}

\begin{rem}
	As pointed out in \cite[Def 13]{carbonipedicchiorosicky01}, any regular category is weakly regular and any weakly regular category with pullbacks has stable regular epimorphisms.
\end{rem}
We will need the following result by Makkai.
\begin{thm}[{\cite[Thm 5.1]{makkai88}}]\label{Makkai}
	Let $\ca$ be a $\kappa$-accessible category and $\III$ a category of cardinality strictly smaller than $\kappa$. Then $[\III,\ca]$ is $\kappa$-accessible and $\Pres_{\kappa}([\III,\ca]) = [\III,\Pres_{\kappa}(\ca)]$.
\end{thm}
The following is the $\kappa$-ary generalization of \cite[Thm 14]{carbonipedicchiorosicky01}.
\begin{lem} \label{regularweaklyregular}
	Let $\CCC$ be a small category with $\kappa$-small colimits. The following are equivalent.
	\begin{enumerate}[label=\rm (\arabic*)]
		\item $\CCC$ is weakly regular.
		\item $\Ind_{\kappa}(\CCC)$ is regular.
	\end{enumerate}
\end{lem}
\begin{proof}
	From \Cref{Makkai} above, one can easily deduce that, given a category $\III$ with a finite set of objects, one has that
	\begin{equation*}
	[\III,\Lex_{\kappa}(\CCC^{\op},\Set)] \cong \Lex_{\kappa}([\III,\CCC^{\op}],\Set). \footnote{Observe that actually more is true. If $\kappa > \aleph_0$ the equivalence holds for every category $\III$ with cardinality smaller than $\kappa$. Nonetheless, for our purposes is enough to consider $\III$ with a finite set of objects.}
	\end{equation*}
	From this fact, it is easily deduced that the category of regular epimorphisms of $\Ind_{\kappa}(\CCC) = \Lex_{\kappa}(\CCC^{\op},\Set)$ is a locally $\kappa$-presentable category presented by the category of regular epimorphisms in $\CCC^{\op}$, and this is just the $\kappa$-version of \cite[Lem 12]{carbonipedicchiorosicky01}. Once this is observed, the proof of the result is then essentially identical to the one provided in \cite{carbonipedicchiorosicky01}.
\end{proof}
\subsection{Exactness and pro-exact syntaxes}\label{exactness}
In \cite[\S9]{carbonipedicchiorosicky01} a syntactical characterization of exact locally finitely presentable categories is provided. In this subsection we present the generalization of this characterization to locally $\kappa$-presentable categories.

Recall (see for example \cite[\S9]{carbonipedicchiorosicky01}) that a \emph{graph} in a category $\CCC$ is a pair of morphisms in $\CCC$
\[
A_1 \overset{r_0}{\underset{r_1}{\rightrightarrows}}A_0
\]
where $A_1$ is the object of edges, $A_0$ is the object of vertices and $r_0$ and $r_1$ are respectively the source and the target morphisms.
A graph $r_0, r_1: A_1 \rightrightarrows A_0$ is called: 
\begin{itemize}
	\item \emph{reflexive} if there exists a morphism $d: A_0 \to A_1$ such that $r_0 d = r_1 d = \id_{A_0}$;
	\item \emph{symmetric} if there is a morphism $s: A_1 \to A_1$ such that $ss=\id_{A_1}$ and $r_0 s = r_1$ (and thus $r_1 s = r_0$);
	\item \emph{transitive} if the pullback $A_1 \times_{A_0} A_1$ exists in $\CCC$ and there is a morphism $t: A_1 \times_{A_0} A_1 \rightarrow A_1$ such that $r_1 \, q_1 = r_1 \, t$ and $r_0 \, q_0 = r_0 \, t$, where $q_0$ and $q_1$ are the natural projections
	\begin{equation*}
	\begin{tikzcd}
	A_1 \times_{A_0} A_1 \arrow[d, "q_0"] \arrow[r, "q_1"] & A_1 \arrow[d, "r_0"] \\
	A_1 \arrow[r, "r_1"] & A_0.
	\end{tikzcd}
	\end{equation*}
\end{itemize}
Given a finitely complete category $\CCC$, an \emph{internal equivalence relation} on an object $X \in \CCC$ (or a \emph{congruence} on $X$) is a reflexive, symmetric and transitive graph in $\CCC$ given by the composition $$R \subseteq X \times X \overset{\pi_0}{\underset{\pi_1}{\rightrightarrows}} X,$$ where $R$ is a subobject of $X \times X$ and $\pi_0, \pi_1: X \times X \rightarrow X$ are the projections.

Recall that a graph $s_0,s_1: B_1 \rightrightarrows B_0$ is an \emph{iteration} of a graph $r_0, r_1: A_1 \rightrightarrows A_0$ if $A_0 = B_0$ and there is a natural number $n$ and $(n+1)$ morphisms $g_0,\ldots, g_n: B_1 \rightarrow A_1$ such that any edge on $B_1$ is given by a path of length $n+1$ in $A_1$, that is, one has that
\[
\begin{aligned}
s_0 &= r_0 g_0,\\
r_1 g_0 &= r_0 g_1,\\
&\,\,\,\vdots\\
r_1 g_{n-1} &= r_0 g_n,\\
s_1 &= r_1 g_n.
\end{aligned}
\]

\begin{defn} \label{defexact}
	A category $\CCC$ is \textit{exact} when it is regular and every internal equivalence relation is effective, or equivalently, every congruence is a kernel pair.
\end{defn}
\begin{defn} \label{defproexact}
	Let $\CCC$ be a weakly regular category with coequalizers. We say $\CCC$ is \textit{pro-exact} when for any reflexive and symmetric graph $r_0, r_1: A_1 \rightrightarrows A_0$, any graph $t_0,t_1: C_1 \rightrightarrows A_0$ such that
	\[
	\coeq(r_0,r_1)t_0 = \coeq(r_0,r_1)t_1	
	\]
	factorizes through a quotient of an iteration of $r_0, r_1: A_1 \rightrightarrows A_0$, i.e. for all such graph $t_0,t_1: C_1 \rightrightarrows A_0$, there exists an iteration $s_0,s_1: B_1 \rightrightarrows A_0$ of $r_0, r_1: A_1 \rightrightarrows A_0$ and a commutative diagram of graphs
	\[
	\begin{tikzcd}[row sep=large, column sep = large]
	B_1 \arrow[r,shift left,"s_0"] \arrow[r,shift right,"s_1"'] \arrow[d,"u"'] &A_0\\
	D_1 \arrow[ru,shift left,"w_0"] \arrow[ru,shift right,"w_1"'] &C_1, \arrow[u,shift left,"t_0"] \arrow[u,shift right,"t_1"'] \arrow[l,"v"] 
	\end{tikzcd}
	\]
	where $u$ is a regular epimorphism.
\end{defn}

\begin{rem}\label{confused}
One should observe that exactness does not imply pro-exactness in general. As pointed out in \cite{carbonipedicchiorosicky01}, an exact category with coequalizers is pro-exact if and only if for any reflexive and symmetric relation, the equivalence relation generated by it is equal to a full iteration of itself for some $n$. Some explicit counterexamples can be found in \cite[\S2]{borceuxpedicchio99}. For example, the denumerable power of the category of finite sets is exact but is not pro-exact. 
\end{rem}

The following is the $\kappa$-ary generalization of \cite[Th. 19]{carbonipedicchiorosicky01}.
\begin{lem} \label{exact} Let $\CCC$ be a small category with $\kappa$-small colimits. The following are equivalent.
	\begin{enumerate}[label=\rm (\arabic*)]
		\item $\CCC$ is pro-exact. 
		\item $\Ind_{\kappa}(\CCC)$ is exact.
	\end{enumerate}
\end{lem}
\begin{proof}
	The proof, based on \Cref{Makkai} and \Cref{regularweaklyregular} above, is essentially identical to the proof of the finitary case as presented in \cite{carbonipedicchiorosicky01}.
\end{proof}

\subsection{Characterization of Grothendieck topoi amongst locally presentable categories}\label{characterizationGrothendieck}

In this subsection we can finally provide a proof of \Cref{characterizationgeneralization}. 
\begin{thm}\label{characterizationkappa}
	Let $\kappa$ be regular cardinal and let $\CCC$ be a small category closed under $\kappa$-small colimits. The following are equivalent:
	\begin{enumerate}[label=\rm (\arabic*)]
		\item $\CCC$ is $\kappa$-extensive and pro-exact.
		\item $\Ind_{\kappa}(\CCC)$ is a Grothendieck topos.
	\end{enumerate}
	A locally $\kappa$-presentable category $\ca$ is hence a Grothendieck topos if and only if its full category of $\kappa$-presentable objects $\Pres_{\kappa}\ca$ is $\kappa$-extensive and pro-exact.
	\begin{proof}
		The statement follows directly from \Cref{Giraud} together with \Cref{extensive} and \Cref{exact}.
	\end{proof}
\end{thm}
This result motivates us to introduce the following:
\begin{defn}\label{prototopos}
	Let $\kappa$ be a regular cardinal. A category $\CCC$ closed under $\kappa$-small colimits is called a \emph{$\kappa$-prototopos} if $\CCC$ is $\kappa$-extensive and pro-exact. 
\end{defn}
\begin{rem}
	The definition of $\aleph_0$-prototopos looks like a weakening of the notion of pretopos (a \emph{pretopos} is an exact and extensive category). One might think that exact $\aleph_0$-prototopoi are pretopoi. This is not true in general in view of \Cref{confused}.
\end{rem}
\begin{exa} A beautiful example of this machinery is the well understood topos $\Set$. The category $\Set$ corresponds to $\Ind_{\kappa}(\Set_{\kappa})$, where by $\Set_{\kappa}$ we denote the category of $\kappa$-small sets. We thus have that $\Set_{\kappa}$ is a $\kappa$-prototopos, it is even an elementary topos if $\kappa$ is strongly inaccessible (see \cite[Ex A2.1.2]{johnstone02-SE1}).  
\end{exa}

\section{Prototopoi-topoi duality}\label{dualitysection}
In order to obtain a restriction to Grothendieck topoi of Gabriel-Ulmer duality, we could just make use of \Cref{characterizationkappa} above, which yields a restriction at the level of objects, and consider the correspoding full sub-$2$-category of $\kappa$-prototopoi in $\Rex_{\kappa}$ and that of $\kappa$-presentable Grothendieck topoi in $\catpresl$. Observe, nonetheless, that the natural $2$-category of locally $\kappa$-presentable Grothen\-dieck topoi to consider if we want to condense both the topos theoretic and the semantic perspectives, should not have as $1$-cells all the pairs of adjoint functors whose right adjoint is $\kappa$-accessible (the relevant morphisms from the syntax-semantic point of view, as shown by Gabriel-Ulmer duality), but only those who are also geometric morphisms (the relevant morphisms to consider in topos theory). 
The main goal of this section is to investigate how we can specialize Gabriel-Ulmer duality at the level of $1$-cells as well, so that we obtain a semantic-topos theoretical duality.  

\subsection{The pseudofunctor $\Ind_{\kappa}$}
As explained in \S\ref{GUduality}, Gabriel-Ulmer duality is provided by the pseudofunctor 
\[\Ind_{\kappa}: \Rex_{\kappa}^{\coop} \rightarrow \catpresl,\]
consisting in taking the $\Ind_{\kappa}$-completion at the level of $0$-cells, which is a locally $\kappa$-presentable category.
In order to understand how we can suitably restrict this pseudofunctor to our case of interest, we first revise its definition and summarize its properties as they were already presented in \S\ref{GUduality} above.

Given two small $\kappa$-cocomplete categories $\AAA$ and $\BBB$ and a $\kappa$-small colimit preserving functor $f: \AAA \rightarrow \BBB$, $\Ind_{\kappa}(f)$ is the adjunction $S(f) \dashv Z(f)$, where 
\begin{itemize}
	\item the left adjoint $S(f)$ is given by the left Kan extension $\text{Lan}_{\alpha_{\AAA}}(\alpha_{\BBB} f)$, and hence we have that the diagram
	\begin{equation}
	\vcenter{\xymatrix{
			\AAA \ar[d]^{ \alpha_{\AAA}}\ar[r]^f  & \BBB \ar[d]_{ \alpha_{\BBB}}\\
			\Ind_{\kappa} (\AAA) \ar[r]^{S(f)}  &\Ind_{\kappa} (\BBB),
	}}
	\end{equation}
	is commutative, where $\alpha_{\AAA}:\AAA \rightarrow \Ind_{\kappa}(\AAA)$ (resp. $\alpha_{\BBB}:\BBB \rightarrow \Ind_{\kappa}(\BBB)$) denotes the corestriction of the Yoneda embedding; 
	\item the right adjoint $Z(f)$ is given by the restriction to $\Ind_{\kappa}(\DDD)$ of the restriction of scalars functor $f^*: [\DDD^{\op},\Set] \to [\CCC^{\op}, \Set]$, that is, we have the commutative diagram
	\begin{equation}
	\begin{tikzcd}
	\Ind_{\kappa} (\AAA) \ar[d,hook] & \Ind_{\kappa} (\AAA) \ar[d,hook] \ar[l,"Z(f)"] \\
	\left[ \AAA^{\op},\Set \right]  &\left[ \BBB^{\op},\Set\right]; \ar[l,"f^*"]
	\end{tikzcd}
	\end{equation}
	\item $S(f)$ is cocontinuous and preserves $\kappa$-presentable objects, which implies that it is \emph{strongly $\kappa$-accessible} (i.e. it is $\kappa$-accessible and preserves $\kappa$-presentable objects).
\end{itemize}  

\subsection{The suitable $2$-categories}
As already mentioned above, our category of interest is the $2$-category $\Grtopoi_{\kappa}$ of locally $\kappa$-presentable Grothendieck topoi defined as follows:
\begin{itemize}
	[leftmargin=+.6in]
	\item[\textbf{$0$-cells}] are the Grothendieck topoi which are locally $\kappa$-presentable;
	\item[\textbf{$1$-cells}] are the geometric morphisms whose right adjoints are $\kappa$-accessible, the direction of which determined by the direction of the right adjoint;
	\item[\textbf{$2$-cells}] are the natural transformations between the right adjoints.
\end{itemize}
\begin{rem}
	$2$-categories of Grothendieck topoi can be often encountered in the literature with a different choice of $2$-cells, namely the natural transformations between the left adjoints. In this paper we stick to the choice of $2$-cells from the classical \cite{SGA4-1}.
\end{rem}

In order to find an appropriate $2$-category of $\kappa$-prototopoi such that the pseudofunctor $\Ind_{\kappa}$ takes values in $\Grtopoi_{\kappa}$, we have to choose as $1$-cells the $\kappa$-small colimit preserving functors $f: \AAA \rightarrow \BBB$ between $\kappa$-prototopoi for which $\Ind_{\kappa}(f)$ with $\kappa$-accessible right adjoint is in addition a geometric morphism, that is, the left adjoint $S(f)$ is left exact. This is a flatness-type property that we encode in the following definition:
\begin{defn}\label{indkappaflat}
	A functor $f: \AAA \to \BBB$ between small categories is called $\Ind_{\kappa}$-flat if $S(f)$ preserves finite limits, or equivalently, if $\text{Lan}_{\alpha_{\AAA}}(\alpha_{\BBB} f)$ preserves finite limits.
\end{defn}

We hence introduce the $2$-category $\Prototopoi_{\kappa}$ of $\kappa$-prototopoi as follows:
\begin{itemize}
	[leftmargin=+.6in]
\item[\textbf{$0$-cells}] are $\kappa$-prototopoi, that is $\kappa$-extensive, pro-exact, small categories with $\kappa$-small colimits;
\item[\textbf{$1$-cells}] are the $\Ind_{\kappa}$-flat functors preserving $\kappa$-small colimits;
\item[\textbf{$2$-cells}] are the natural transformations.
\end{itemize}

Observe that $\Ind_{\kappa}$-flatness has been introduced as an \emph{ad hoc} definition in order to obtain our desired result. We will see in \Cref{cover flat} that there is a natural topological interpretation of this notion. For the moment, let's just analyse its relation of the classical notion of flatness. Recall that a functor $f: \AAA \to \BBB$ is said to be \emph{flat} if and only if the left Kan extension $\Lan_f \coloneqq \Lan_{y_{\AAA}} (y_{\BBB}f): [\AAA^{\op},\Set] \rightarrow [\BBB^{\op},\Set]$ preserves finite limits. It is easy to see that $\Lan_f$ is the left adjoint of the restriction of scalars functor $f^*:[\BBB^{\op},\Set] \rightarrow [\AAA^{\op},\Set]$. The following shows that for a $\kappa$-small colimit preserving functor between $\kappa$-prototopoi, the notion of $\Ind_{\kappa}$-flatness is weaker than the notion of flatness.
\begin{prop} \label{flatnesscomparison}
	Let  $f: \AAA \to \BBB$ be a flat functor that preserves $\kappa$-small colimits between $\kappa$-prototopoi. Then the  adjunction $\Ind_{\kappa}(f)$ given by $S(f): \Ind_{\kappa}(\AAA) \rightleftarrows \Ind_{\kappa}(\BBB): Z(f)$ is a geometric morphism whose right adjoint is accessible.
\end{prop}
\begin{proof}
	As a consequence of \Cref{GU} we only have to prove that  $S(f)$ is left exact, or equivalently that $f$ is $\Ind_{\kappa}$-flat.
	By \Cref{characterizationkappa}, $\Ind_{\kappa}(\AAA)$ is a Grothendieck topos which can be recovered as the category of sheaves of the site given by the $\kappa$-extensive topology on $\AAA$. Consequently, we have that the natural embedding $i_{\AAA}: \Ind_{\kappa}(\AAA) \hookrightarrow [\AAA^{\op},\Set]$ has a left adjoint $L_{\AAA}: [\AAA^{\op},\Set] \rightarrow \Ind_{\kappa}(\AAA)$ which is left exact. We have the following diagram of functors
	%
	\[\vcenter{\xymatrix{
			\AAA \ar[dr]^{ \alpha_{\AAA}}\ar[rrr]^f \ar[dd]_{y_{\AAA}} &&& \BBB \ar[dd]^{y_{\BBB}}\ar[dl]_{ \alpha_{\BBB}}\\
			& \Ind_k (\AAA) \ar@<2.5pt>[dl]^{i_{\AAA}} \ar@<2.5pt>[r]^{ S(f)}&\ar@{>}@<2.5pt>[l]^{ Z(f)}  \Ind_k (\BBB) \ar@<-2.5pt>[dr]_{i_{\BBB}} & \\
			[\AAA^{\op},\Set] \ar@<2.5pt>[ur]^{L_{\AAA}} \ar@<2.5pt>[rrr]^{\text{Lan}_f}&&&\ar@<2.5pt>[lll]^{f^*} \ar@<-2.5pt>[ul]_{L_{\BBB}}   [\BBB^{\op},\Set],
	}}
	\]
	Observe that 
	\begin{equation}
	S(f)  \, L_{\AAA} \cong  L_{\BBB} \, \text{Lan}_f.
	\end{equation}
	Indeed, this follows from:
	\begin{align*}
	S(f) \, L_{\AAA} &\cong \text{Lan}_{y_{\AAA}}(\alpha_{\BBB}f) \\
	& \cong \text{Lan}_{y_{\AAA}}(L_{\BBB}y_{\BBB}f) \\
	& \cong L_{\BBB} \, \text{Lan}_{y_{\AAA}}(y_{\BBB}f) \\
	& \cong  L_{\BBB} \, \text{Lan}_f , \\
	\end{align*}
	where the nontrivial isomorphism $S(f) \, L_{\AAA} \cong \text{Lan}_{y_{\AAA}}(\alpha_{\BBB}f)$ holds because $S(f) \, L_{\AAA}$ is cocontinous and the Yoneda embedding $y_{\AAA}$ is a dense functor, and thus \[ S(f) \, L_{\AAA} \cong  \text{Lan}_{y_{\AAA}}(S(f)L_{\AAA}y_{\AAA}) \cong \text{Lan}_{y_{\AAA}}(S(f)\alpha_{\AAA}) \cong  \text{Lan}_{y_{\AAA}}(\alpha_{\BBB}f).\]
	Consequently, we have that 
	\begin{equation}\label{Sfpresheaves}
	S(f) \cong S(f)  \, L_{\AAA} \, i_{\AAA} \cong L_{\BBB} \, \text{Lan}_f \, i_{\AAA}.
	\end{equation}
	Now observe that $i_{\AAA}$ is a right adjoint and thus preserves all limits, $\text{Lan}_f$ preserves finite limits because $f$ is flat, and the functor $L_{\BBB}$ preserves finite limits by the discussion above. We can thus conclude that $S(f)$ preserves finite limits as desired.
\end{proof}
\begin{rem}
	As just shown, flat $\kappa$-small colimit preserving functors between $\kappa$-prototopoi produce under the pseudofunctor $\Ind_{\kappa}$ geometric morphisms with $\kappa$-accessible right adjoints. One could hope that given a geometric morphism with $\kappa$-accessible right adjoint $L: \Ind_{\kappa}(\AAA) \rightleftarrows \Ind_{\kappa}(\BBB): R$ the restriction $f$ of the left adjoint $L$ to the $\kappa$-presentable objects would be flat. If we could infer from the left-exactness of $L$, which equals $S(f)$, the left-exactness of $\text{Lan}_f$, the proof would be complete. Unfortunately, this implication is not true in general. This shows that the the notion of flatness for a $\kappa$-small colimit preserving functors between $\kappa$-prototopoi is strictly stronger than the notion of $\Ind_{\kappa}$-flatness, which is the one required to obtain the desired Gabriel-Ulmer for locally $\kappa$-presentable Grothendieck topoi.
\end{rem}

\subsection{Gabriel-Ulmer duality for Grothendieck topoi}
We are finally able to provide the desired restriction of Gabriel-Ulmer duality. 
\begin{thm}\label{1} Let $\kappa$ be a regular cardinal.
There is a biequivalence of categories $$\Ind_{\kappa}: \Prototopoi_\kappa^{\coop} \rightleftarrows  \Grtopoi_{\kappa} : \Pres_{\kappa}. $$ 
\end{thm}
\begin{proof}
This is a direct consequence of \Cref{characterizationkappa} and Gabriel-Ulmer duality (\Cref{GU}), after observing that the property of $1$-cells being $\Ind_{\kappa}$-flat at the syntax side corresponds under the duality to $1$-cells with a left exact left adjoint at the semantics side, which is a direct consequence of the definition of $\Ind_{\kappa}$-flatness (\Cref{indkappaflat}).
\end{proof}

\section{Cofinality}\label{cofinality}
While the prototopoi-topoi duality from \Cref{1} above provides a good understanding of the $2$-category $\Grtopoi_{\kappa}$ for any $\kappa \in \RegCard$, from a topos theoretical point of view our interest is still focused on the bigger $2$-category $\Grtopoi$ of Grothendieck topoi with geometric morphisms and natural transformations between the right adjoints. In this regard, the family of $2$-categories $\{\Grtopoi_{\kappa}\}_{\kappa \in \RegCard}$ already provides useful information to understand $\Grtopoi$, as it is a \textit{cofinal} family in $\Grtopoi$. This is, given any geometric morphism \[(f: \ca \rightleftarrows \cb: g) \in \Grtopoi(\cb, \ca)\] 
we have that there exists a regular cardinal $\kappa$ such that $(f \dashv g) \in \Grtopoi_{\kappa}(\cb, \ca)$ (see \cite[Prop. 2.23]{adamekrosicky94}). The same statement stays true when replacing ``geometric morphism'' with ``small diagram'' because of the remark under \cite[Prop. 2.19]{adamekrosicky94}.
\begin{rem}\label{essentiallysmall}
	Based on this observation, one might be tempted to think that given two Grothendieck topoi $\ca, \cb$ there exists a regular cardinal $\kappa$ such that $\Grtopoi(\cb, \ca) = \Grtopoi_{\kappa}(\cb, \ca).$ This is actually not true; while by \Cref{1} we have that $\Grtopoi_{\kappa}(\cb, \ca) \cong \Prototopoi_{\kappa}^{\coop}(\Pres_{\kappa}(\cb),\Pres_{\kappa}(\ca))$ is always essentially small, $\Grtopoi(\cb, \ca)$ is not essentially small in general. For example, we have that $\Grtopoi(\Set, [C^{\op},\Set]) \cong \Ind(C)$. 
\end{rem}
Observe that, in any case, we do have natural embeddings $\Grtopoi_{\alpha} \subseteq \Grtopoi_{\beta}$ for any $\alpha \leq \beta$ regular cardinals, such that
\begin{equation}
	\Grtopoi = \bigcup_{\kappa \in \RegCard}\Grtopoi_{\kappa}.
\end{equation}
Calling \textit{height} of a Grothendieck topos $\cg$ the smallest cardinal $\kappa$ such that  $\cg$ is locally $\kappa$-presentable, it would be very interesting to have a criterion that detects the height of a Grothendieck topos in this directed chain of embeddings. In what follows we point out a way of detecting an upper bound for the height.
\begin{defn}\label{defnsize}
	Let $\CCC$ be a small category, $\tau$ a Grothendieck topology on $\CCC$ and $\kappa$ a regular cardinal. We say that $\tau$ is a \emph{weakly $\kappa$-ary Grothendieck topology} if every covering sieve in $\tau$ contains a $\kappa$-small generating family. Then we say that $(\CCC, \tau)$ is a \emph{weakly $\kappa$-ary Grothendieck site}.
\end{defn}
\begin{rem}\label{weaklykary}
	 We borrow this terminology from \cite{shulman12}. More concretely, we say that a Grothendieck topology in $\CCC$ is a weakly $\kappa$-ary Grothendieck topology if it is the Grothendieck topology associated via \cite[Prop. 3.4]{shulman12} to a weakly $\kappa$-ary topology in $\CCC$ in the sense of \cite[Def. 3.1]{shulman12}. As throughout the whole paper we only work with Grothendieck topologies and sites there is no risk of confusion in dropping ``Grothendieck'' from ``weakly $\kappa$-ary Grothendieck topology'' and ``weakly $\kappa$-ary Grothendieck site'', and hence we will do so for the sake of brevity.  
\end{rem}
\begin{rem}\label{coverage}
	Observe that a site $(\CCC,\tau)$ is weakly $\kappa$-ary if and only if the $\kappa$-small covering families form a coverage in the sense of \cite[Def. C2.1.1]{johnstone02-SE2} determining $\tau$, i.e. the smallest Grothendieck topology containing this coverage is $\tau$.
\end{rem}
\begin{prop}\label{criterion} Let $\cg$ be a Grothendieck topos. Then the following are equivalent:
\begin{enumerate}[label=\rm (\arabic*)]
\item $\cg$ is $\kappa$-locally presentable.
\item There exists a weakly $\kappa$-ary site $(\CCC,\tau)$ such that $\Sh(\CCC,\tau) \cong \cg$.
\end{enumerate}
\begin{proof}
	\begin{itemize}
		\setlength{\itemindent}{1em}
		\item[ ] $\,$
		\item[$(1) \Rightarrow (2)$] As a direct consequence of the Lemme de comparaison \cite[Expos\'e iii, Thm 4.1]{SGA4-1}, we know that $\cg$ can be recovered, up to equivalence, as $\Sh(\Pres_{\kappa}(\cg),\tau)$, where $\tau$ is the Grothendieck topology induced by the canonical topology in $\cg$ via the dense embedding $\iota: \Pres_{\kappa}(\cg) \hookrightarrow \cg$. Observe that $\tau$ is a subcanonical topology and that covering families in $\tau$ are the families $\{A_i \rightarrow A\}_{i\in I}$ such that $\{h_{A_i}=\iota(A_i) \rightarrow h_A=\iota(A)\}_{i\in I}$ is jointly-epimorphic in $\Sh(\Pres_{\kappa}(\cg),\tau) \cong \cg$. Due to the fact that arbitrary coproducts of representables in $\Sh(\Pres_{\kappa}(\cg),\tau)$ can be written as $\kappa$-filtered colimits of $\kappa$-small coproducts of representables, one can easily conclude that every covering sieve of $\tau$ is contains a $\kappa$-small covering family. 
		\item[$(2) \Rightarrow (1)$] Let $(\CCC, \tau)$ a weakly $\kappa$-ary site of definition of $\cg$. By \Cref{coverage} we know that the $\kappa$-small covering families form a coverage determining $\tau$, and hence, an element $F \in [\CCC^{\op},\Set]$ belongs to $\Sh(\CCC,\tau)$ if and only if, for every $C \in \CCC$ and for every $\kappa$-small covering family $\{f_i: C_i \rightarrow C\}$, we have that $F(C) = \lim F(C_i)$ (see \cite[Def. C2.1.2]{johnstone02-SE2}). As in the proof of \cite[Prop. 5.2.1]{makkaipare89}, we can consider the sketch $\cs$ in the category $\CCC^{\op}$ with the set of limit cones given by the $\kappa$-small covering families $\{f_i^{\op}: C \rightarrow C_i\}$ and the set of colimit cocones empty. Then it is immediate by the characterization of sheaves just provided, that $\Sh(\CCC,\tau)$ is equivalent to the category $\Mod \cs$ of models of the limit sketch $\cs$. Then, by \cite[Cor. 1.52 + Rem]{adamekrosicky94} or \cite[Thm D2.3.6]{johnstone02-SE2}, we can conclude that $\Sh(\CCC,\tau)$ is locally $\kappa$-presentable. 
	\end{itemize}
\end{proof}
\end{prop}

\begin{rem}\label{kappasmall}
Observe that for a weakly $\kappa$-ary site $(\CCC,\tau)$ we have that the essential image of $\CCC$ via the functor $\CCC \rightarrow [\CCC^{\op}, \Set] \rightarrow \Sh(\CCC,\tau)$ is a strong generator of $\Sh(\CCC,\tau)$ formed by locally $\kappa$-presentable objects (see \cite[Thm. 1.46]{adamekborceuxlackrosicky02}). Then by \cite[Thm 7.2]{kelly82} we can conclude that $\Ind_{\kappa}(\CCC) \cong \Lex_\kappa(\hat{\CCC}) \cong \Sh(\CCC,\tau)$, where $\hat{\CCC}$ is the closure of $\CCC$ under $\kappa$-small colimits in $\Sh(\CCC,\tau)$.
\end{rem}

Quite often it is very useful from an operative point of view to have cartesian sites, that is, sites closed under finite limits. The previous proposition can be refined in this direction, paying the price of a less elegant statement. 
\begin{lem}\label{cartesian}
Given a $\kappa$-presentable topos $\cg$ there exists a $\lambda \geq \kappa$ and a \textit{cartesian} weakly $\lambda$-ary site of definition of $\cg$.
\end{lem}
\begin{proof}
This is equivalent to prove that given a locally $\kappa$-presentable topos $\cg$, we can find a $\lambda \geq \kappa$ such that $\Pres_{\lambda}(\cg)$ is closed under finite limits. This follows directly from \cite[Prop 5.2]{makkai88}.

\end{proof}

\begin{rem}
Applying this result we hence have that the $\lambda$-prototopos $\Pres_{\lambda}(\cg)$ has finite limits. The intuition suggests that prototopoi with finite limits might be a much simpler object to describe, even if at the moment we did not find a clean axiomatization that does not involve pro-exactness. Unfortunately one cannot hope for any improvement of the topoi-prototopoi duality because there is a mismatch between the height of the topos and the cardinal $\lambda$ needed to saturate presentable objects with finite limits.
\end{rem}

\section{Sites and syntaxes}\label{sitesandprototopoi}
Grothendieck topoi are a powerful tool in order to discover connections between different mathematical theories, as the program ``toposes as bridges'' developed by Caramello \cite{caramello18} has extensively shown. The philosophy behind this program relies on seeing Morita equivalences of theories (i.e. equivalences of their classifying topoi) as a bridge to transport properties and results from one theory to another. In this way, properties of a certain topos can incarnate apparently unrelated properties of different representing sites, unraveling connections between mathematical theories hidden at first sight. This theory is fundamented on the fact that there is no privileged small site to represent a topos, as already brought to attention in \cite[Expos\'e iv, Rem 1.3]{SGA4-1}. Formally, this lack of a canonical small site to represent a topos translates into the following statement:
\begin{prop}\label{reflection}
The pseudofunctor $$\Sh: \Sites^{\coop} \to \Grtopoi $$ is not the left biadjoint of a bi-reflection, where $\Sites$ denotes the $2$-category of small sites with morphisms of sites and natural transformations and the pseudofunctor $\Sh$ sends each site to its category of sheaves, each morphism of sites to the induced geometric morphism and each natural transformation to the induced natural transformation between the right adjoints.
\end{prop}
\begin{proof} 
As already pointed out in \Cref{essentiallysmall}, the $2$-category $\Grtopoi$ is not locally small. Indeed, for any small category $\CCC$ we have that $\Grtopoi(\Set, [\CCC,\Set]) \cong \Ind(\CCC)$. Therefore, $\Grtopoi$ cannot be bi-reflective in the $2$-category $\Sites^{\coop}$ which is locally small.
\end{proof}
\begin{rem}
	A detailed description of the pseudofunctor $\Sh$ can be found for example in \cite[\S4]{ramosgonzalez18} for the setup enriched over categories of modules. This description works as well for the classical non-enriched setup.
\end{rem}

Nevertheless, for practical reasons, one would still want to have some sort of functorial ``dictionary'' to translate problems in $\Grtopoi$ to problems in $\Sites$, where one can work much more comfortably, and be able to translate the results back to $\Grtopoi$. 

To some extent, one approach to achieve this goal is by using the biequivalence
$$\Sites[\text{LC}^{-1}] \to \Grtopoi^{\coop}$$
from \cite{ramosgonzalez18} between $\Grtopoi^{\coop}$ and the bicategory of fractions in the sense of \cite{pronk96} of the 2-category $\Sites$ along a suitable class of $1$-morphisms called LC morphisms. This result allows to translate properties in $\Sites$ which are compatible with LC morphisms to properties in $\Grtopoi$. This was for example the approach in \cite{lowenramosgonzalezshoikhet,ramosgonzalez18} in order to define a monoidal structure in $\Grtopoi$ by introducing an LC-compatible monoidal structure in $\Sites$. 

In this section we provide another approach to this goal. More concretely, we relate the prototopoi presentations of Grothendieck topoi obtained from \Cref{ourduality} with the site presentations in order to show that though a bireflection of $\Grtopoi$ in $\Sites$ does not exist, it actually does exist in a certain ``stratified way'', more precisely, we show that there is a bireflection 
$$\Grtopoi_{\kappa} \hookrightarrow \Sites_{\kappa}^{\coop}$$
between the full 2-subcategory $\Grtopoi_{\kappa}$ of $\Grtopoi$ given by the $\kappa$-presentable Grothendieck topoi and the conjugate-opposite of the full 2-subcategory $\Sites_{\kappa}$ of $\Sites$ given by the weakly $\kappa$-ary sites. 

\begin{rem}\label{sizes}
It is natural to object that it should be possible to bypass the size issues in \Cref{reflection}. The price one then needs to pay is that of allowing big sites (that yet have to admit a small topologically generating set). Then, as observed in \cite[Expos\'e iv, Rem 1.3]{SGA4-1}, there is a canonical choice of representing (big) site for a given Grothendieck topos, namely itself endowed with the canonical topology. This escapes our initial motivations, especially if we want to exploit the syntax-semantic aspect of the theory and think about sites as syntactic presentations of Grothendieck topoi.
\end{rem}

\subsection{Previous topological considerations}\label{topologicalconsiderations}
Given a $\kappa$-prototopos $\AAA$, we have that $\Ind_{\kappa}(\AAA)$ is a locally $\kappa$-presentable Grothendieck topos and $\AAA$ is a full subcategory therein which is in addition a strong generator. Then, as a consequence of Giraud's theorem (see \cite[Expos\'e iv, Cor 1.2.1]{SGA4-1}), there is a topology $\tau$ on $\AAA$ such that $\Ind_{\kappa}(\AAA) \cong \text{Sh}(\AAA, \tau)$. In particular, the localization 
\[L_{\AAA}: [\AAA^{\op},\Set] \rightleftarrows  \Ind_{\kappa}(\AAA): i_{\AAA}\] 
is a geometric functor, i.e. $L_{\AAA}$ preserves finite limits.

We can actually describe $\tau$ precisely, as it is the topology induced in $\AAA$ by the canonical topology in $\Ind_{\kappa}(\AAA)$ through the natural embedding $\iota_{\AAA}: \AAA \hookrightarrow \Ind_{\kappa}(\AAA)$. More precisely, a family of maps $\{A_i \to A \}$ is a cover of $A$ in the topology $\tau$ if and only if the induced family $\{\iota(A_i) \to \iota(A) \}$ is a jointly epimorphic family in $\Ind_{\kappa}(\AAA)$, or equivalently $\coprod_i \iota(A_i) \to \iota(A)$ is an epimorphism in $\Ind_{\kappa}(\AAA)$. We shall denote this topology from now on as $\tau_{\kappa}^{\AAA}$. Observe that $\tau_{\kappa}^{\AAA}$ is weakly $\kappa$-ary, as already pointed out in the proof of \Cref{criterion} above. 

These observations allow us to provide the following topological interpretation of the notion of $\Ind_{\kappa}$-flat morphism between prototopoi. 
\begin{prop} \label{cover flat}
	The $1$-cells $f \in \Prototopoi_{\kappa}(\AAA,\BBB)$ are precisely the morphisms of sites in the sense of \cite[Expos\'e iv, \S4.9.1]{SGA4-1} between the sites $f: (\AAA, \tau_\kappa^{\AAA}) \to (\BBB, \tau_\kappa^{\BBB})$, i.e. the continuous functors $f: (\AAA, \tau_\kappa^{\AAA}) \to (\BBB, \tau_\kappa^{\BBB})$ inducing a geometric morphism between the sheaf categories, or equivalently, the covering-flat cover-preserving functors $f: (\AAA, \tau_\kappa^{\AAA}) \to (\BBB, \tau_\kappa^{\BBB})$. 
	\begin{proof}
		Assume first that $f:\AAA \rightarrow \BBB$ is a morphism of $\kappa$-prototopoi, i.e. it preserves $\kappa$-small colimits and it is $\Ind_{\kappa}$-flat. Taking into account that $\Ind_{\kappa}(\AAA)$ is a topos and hence regular, the fact that $f$ is cover-preserving follows immediatly from the definition of $S(f): \Ind_{\kappa}(\AAA) \rightarrow \Ind_{\kappa}(\BBB)$ and the fact that it preserves colimits. From \cite[Prop 4.16]{shulman12} we know that $f$ is covering-flat if and only if $L_{\BBB} \Lan_f: [\AAA^{\op},\Set] \rightarrow \Ind_{\kappa}(\BBB)$ preserves finite limits. But, as seen in (\ref{Sfpresheaves}) above, we have that $L_{\BBB} \Lan_f \cong S(f) L_{\AAA}$ and $S(f)$ preserves finite limits because $f$ is $\Ind_{\kappa}$-flat and $L_{\AAA}$ preserves small limits because $L_{\AAA} \dashv i_{\AAA}$ is a geometric functor, which concludes the argument. 
		
		If we now assume that $f: \AAA \rightarrow \BBB$ is covering-flat and cover-preserving, we have as a direct consequence of \cite[Prop 4.16]{shulman12} that $S(f)$ is left exact, and hence $f$ is $\Ind_{\kappa}$-flat. The fact that $f$ is $\kappa$-small colimit preserving follows immediately from the fact that both corestricitions of the Yoneda embedding $\AAA \hookrightarrow \Ind_{\kappa}(\AAA)$ and $\BBB \hookrightarrow \Ind_{\kappa}(\BBB)$ preserve $\kappa$-small colimits, together with the fact that $S(f)$ preserves all small colimits.  
	\end{proof}
\end{prop}

\begin{rem}
	In the literature a weaker notion of morphism of sites than that from \cite[Expos\'e iv, \S4.9.1]{SGA4-1} is often considered. Namely, a morphism of sites can be often found defined as a cover-preserving functor which is also flat (or in other terminology, \emph{representably-flat}, which is frequently used in this context in order to emphasize the diference with covering-flatness). The notion we work with is slightly stronger that this, as explained for example in \cite{shulman12}. However, one can easily see that in the case above both notions are equivalent because $\tau_\kappa^{\BBB}$ is a subcanonical topology on $\BBB$.
\end{rem}

\begin{rem}
Observe that when we restrict to cartesian $\lambda$-protopoi, a $1$-cell of prototopoi is precisely a functor preserving finite limits and $\lambda$-small colimits. In \Cref{cartesian} above we showed that one can always assume to be in this case in the daily practice (by considering $\lambda$ a big enough regular cardinal).
\end{rem}

\subsection{$\Grtopoi_{\kappa}$ is bireflective in $\Sites_{\kappa}$} \label{bireflectiontopoisites}
Let $\kappa$ be a regular cardinal. As we already pointed out above, we denote by $\Sites_{k}$ the full sub-$2$-category of the $2$-category of $\Sites$ with objects given by the weakly $\kappa$-ary small sites. Namely, $\Sites_{\kappa}$ is defined as follows:
\begin{itemize}
	[leftmargin=+.6in]
\item[\textbf{$0$-cells}] are weakly $\kappa$-ary small sites $(\CCC, \tau)$; 
\item[\textbf{$1$-cells}] are morphisms of sites (as considered in \Cref{cover flat} above);
\item[\textbf{$2$-cells}] are natural transformations.
\end{itemize}
We are going to show that it is possible to fit the topoi-prototopoi duality (\Cref{1}) in the following diagram of pseudofunctors
\[
\begin{tikzcd}
\Prototopoi^{\coop}_{\kappa} \ar[bend right=11, ddr, "\Ind_{\kappa}"' description] \ar[bend left=11, rr, "\mathrm{T}_{\kappa}" description]  & & \Sites^{\coop}_{\kappa} \ar[bend left=14, ddl, "\mathrm{S}_{\kappa}" description] \\
& & \\
& \Grtopoi_{\kappa}  \ar[bend right=14, uul, "\Pres_{\kappa}" description]  & \\
\end{tikzcd}
\]
where we still need to describe the pseudofunctors $\mathrm{S}_{\kappa}$ and $\text{T}_{\kappa}$.

Let's start with the definition of $\mathrm{S}_{\kappa}$. Recall that we denote by $\Sites$ the $2$-category of sites with morphisms of sites and natural transformations between them and by $\Grtopoi$ the $2$-category of Grothendieck topoi with geometric morphisms and natural transformations between their right adjoints. We have the pseudofunctor $\Sh: \text{Site}^{\coop} \to \Grtopoi$ from \Cref{reflection} defined at the level of objects by taking the category of sheaves. The pseudofunctor $\mathrm{S}_{\kappa}$ is constructed as the restriction to $\text{Site}_\kappa^{\coop}$ and corestriction to $\Grtopoi_{\kappa}$ of the pseudofunctor $\Sh$. More precisely, $\mathrm{S}_{\kappa}$ is defined as follows:
\begin{itemize}
	[leftmargin=+.6in]
\item [\textbf{$0$-cells}] it maps a weakly $\kappa$-ary site $(\CCC, \tau)$ to $\Sh(C,\tau)$, which is $\kappa$-presentable Grothen\-dieck topos as a direct consequence of \Cref{criterion}.
\item[\textbf{$1$-cells}]  it maps a morphism of sites $f: (\CCC,\tau_{\CCC}) \to (\DDD,\tau_{\DDD})$ in $\text{Site}_{\kappa}$ to the geometric morphism $f^s : \Sh(\CCC,\tau_{\CCC}) \rightleftarrows \Sh(\DDD,\tau_{\DDD}): f_s$ (we follow the notations from \cite[Expos\'e iii, \S1]{SGA4-1}). One just needs to verify that $f_s$ is $\kappa$-accessible. Observe that, because of \Cref{kappasmall}, we have essential image of the functor 
$$\alpha_{\CCC}:\CCC \xrightarrow{y_{\CCC}} [\CCC^{\op},\Set] \xrightarrow{L_{\CCC}} \Sh(\CCC,\tau_{\CCC})$$ 
(resp. $\alpha_{\DDD}: \DDD \rightarrow \Sh(\DDD,\tau_{\DDD})$) is a strong generator of $\Sh(\CCC,\tau_{\CCC})$ (resp. $\Sh(\DDD,\tau_{\DDD})$) formed by $\kappa$-presentable objects. In addition, we have that the diagram
\[
\begin{tikzcd}
\CCC \arrow[d] \arrow[d, "\alpha_{\CCC}"] \arrow[r, "f"] & \DDD \arrow[d, "\alpha_{\DDD}"] \\
{\Sh(\CCC,\tau_{\CCC})} \arrow[r, "f^s"] & {\Sh(\DDD,\tau_{\DDD})}
\end{tikzcd}
\]
is commutative up to isomorphism by \cite[Expos\'e iii, Prop 1.2]{SGA4-1} from which, as $f^s$ preserves colimits, one deduces that $f^s$ preserves $\kappa$-presentable objects. This proves that $f_s$ is $\kappa$-accessible.
\item[\textbf{$2$-cells}] it maps every $\alpha: f \Rightarrow g : (\CCC,\tau_{\CCC}) \to (\DDD,\tau_{\DDD})$ to the natural transformation  $\alpha_s: g_s \Rightarrow f_s$ defined by $$(\alpha_s)_F(C) = F(\alpha_C) : g_s(F)(C) = F(g(C)) \to F(f(C)) = f_s(F)(C)$$ for all $C \in \CCC$ and all $F \in \Sh(\DDD,\tau_{\DDD})$. 
\end{itemize}
Let's now define the $2$-functor $\mathrm{T}_{\kappa}$ as follows:
\begin{itemize}
	[leftmargin=+.6in]
\item [\textbf{$0$-cells}]  it maps a $\kappa$-prototopoi $\CCC$ to $(\CCC, \tau_{\kappa}^{\CCC})$, where $\tau_\kappa^{\CCC}$ is defined as in \S\ref{topologicalconsiderations};
\item[\textbf{$1$-cells}]  it is the identity on $1$-cells. This is indeed well-defined as a direct consequence of \Cref{cover flat} above.
\item[\textbf{$2$-cells}] it is the identity on $2$-cells. 
\end{itemize}

\begin{thm}\label{semantic}
The $2$-category $\Grtopoi_{\kappa}$ is bi-reflective in $\Sites_{\kappa}^{\coop}$ via the following biadjunction,
$$\mathrm{S}_{\kappa}:  \Sites_{k}^{\coop} \rightleftarrows  \Grtopoi_{\kappa}:  \mathparagraph_{\kappa}, $$
where $\mathparagraph_{\kappa}$ is given by the composition $\mathrm{T}_{\kappa} \Pres_{\kappa}$.
\end{thm}
\begin{proof}
 It is easy to check that $\mathrm{S}_{\kappa}\, \mathparagraph_{\kappa} \cong \Id_{\Grtopoi_{\kappa}}$. Indeed, by the definitions above it follows that $\mathrm{S}_{\kappa}  \, \mathrm{T}_{\kappa} \cong \Ind_{\kappa}$, and thus $\mathrm{S}_{\kappa}  \, \mathrm{T}_{\kappa} \, \Pres_{\kappa} \cong \Ind_{\kappa} \, \Pres_{\kappa} \cong \Id_{\Grtopoi_{\kappa}}$, where the last pseudonatural equivalence comes from \Cref{1}. This gives us the counit $\epsilon: \mathrm{S}_{\kappa} \, \mathparagraph_{\kappa} \Rightarrow \Id_{\Grtopoi_{\kappa}}$ of the biadjunction. 

We proceed now to construct the unit $\eta: \Id_{\Sites_{\kappa}} \Rightarrow \mathparagraph_{\kappa} \, \mathrm{S}_{\kappa}$.
Let $(\CCC,\tau)$ be weakly $\kappa$-ary site. Then we have by \Cref{criterion} that $\Sh(\CCC,\tau)$ is $\kappa$-presentable. Let's denote by $\DDD = \Pres_{\kappa}(\Sh(\CCC,\tau)) \subseteq \Sh(\CCC,\tau)$. Then, by \Cref{kappasmall}, we know that $L_{\CCC} \, y (C) \in \DDD$ for all $C \in \CCC$. We can hence consider the following functor:
\begin{equation}
	s_{\CCC}: \CCC \to \DDD: C \mapsto L_{\CCC} \, y_{\CCC} (C).
\end{equation}
We show that this functor is a morphism of sites between the sites $(\CCC,\tau)$ and $\mathparagraph_{\kappa} \, \mathrm{S}_{\kappa} (\CCC, \tau_{\CCC}) = (\DDD,\tau_{\DDD})$, and hence we can define 
\begin{equation}
\eta_{(\CCC,\tau_{\CCC})} = s_{\CCC}:(\CCC,\tau_{\CCC}) \to \mathparagraph_{\kappa} \, \mathrm{S}_{\kappa} (\CCC, \tau_{\CCC}).
\end{equation}
Indeed, we have the following commutative diagram
\begin{equation}\label{siteandpresentables}
\begin{tikzcd}[column sep=10ex]
\CCC \ar[r, "s_{\CCC}"] \ar[d, "y_{\CCC}"']  & \DDD \ar[d, "y_{\DDD}"] \ar[hook, bend left = 50, dd] \\
\left[ \CCC^{\op},\Set\right]  \ar[r,"\Lan_{s_{\CCC}}"] \ar[d,"L_{\CCC}"']  & \left[ \DDD^{\op},\Set\right]  \ar[d,"L_{\DDD}"] \\
\Sh(\CCC,\tau_{\CCC})  \arrow{r}{(\epsilon_{\Sh(\CCC,\tau_{\CCC})})^{-1}}[below]{\cong}  & \Sh(\DDD,\tau_{\DDD}),\\
\end{tikzcd}
\end{equation}
from which one concludes that $s_{\CCC}$ is a continuous morphism (using the characterization of continuous morphisms provided by \cite[Expos\'e iii, Prop 1.2]{SGA4-1}) and it induces an equivalence between the sheaf categories, hence it is a morphism of sites. We can even say further, $s_{\CCC}$ is a \emph{special cocontinuous} morphism in the sense of \cite[Tag 03CG]{stacksproject}, as a consequence of a generalization of the Comparison Lemma \cite{kockmoerdijk91} and the fact that $(\DDD,\tau_{\DDD})$ is subcanonical. This argument can also be found in \cite[Tag 03A1]{stacksproject}.

Given now a morphism of sites $f:(\AAA,\tau_{\AAA})\to (\BBB,\tau_{\BBB})$, we have the following diagram
\[
\begin{tikzcd}
(\AAA,\tau_{\AAA}) \ar[rrr, "f"] \ar[dd, "\alpha_{\AAA}"'] \ar[dr,"\eta_{(\AAA,\tau_{\AAA})}"] &&&(\BBB,\tau_{\BBB}) \ar[dd, "\alpha_{\BBB}"] \ar[dl,"\eta_{(\BBB,\tau_{\BBB})}"']\\
&\mathparagraph_{\kappa} \, \mathrm{S}_{\kappa} (\AAA, \tau_{\AAA})  \ar[r,"\mathparagraph_{\kappa}\, \mathrm{S}_{\kappa}(f)"] \ar[dl, hook]  &\mathparagraph_{\kappa} \, \mathrm{S}_{\kappa}(\BBB,\tau_{\BBB}) \ar[dr, hook]&\\
\Sh(\AAA,\tau_{\AAA})  \ar[rrr, "f^s"]  &&& \Sh(\BBB,\tau_{\BBB}),
\end{tikzcd}
\]
where the outter square is commutative up to a canonical isomorphism 
\begin{equation}\label{naturalisom}
\Phi_{f}: \alpha_{\BBB} \, f \overset{\sim}{\Rightarrow} f^s \, \sigma_{\AAA},
\end{equation}
as a consequence of \cite[Expos\'e iii, Prop 1.2]{SGA4-1}. Observe, in addition, that
$$\epsilon_{\Sh(\AAA,\tau_{\AAA})}  \, \alpha_{\mathparagraph_{\kappa}(\mathrm{S}_{\kappa}(\AAA,\tau_{\AAA}))}: \mathparagraph_{\kappa} \, \mathrm{S}_{\kappa}(\AAA,\tau_{\AAA}) \hookrightarrow  \mathrm{S}_{\kappa}(\AAA,\tau_{\AAA})$$
coincides with the canonical embedding of the full subcategory $\Pres_{\kappa}(\Sh(\AAA,\tau_{\AAA})) = \mathparagraph_{\kappa} \, \mathrm{S}_{\kappa}(\AAA,\tau_{\AAA})$ in $\Sh(\AAA,\tau_{\AAA})$ that appears in the diagram, and hence the left triangle of the diagram is commutative because of the commutativity of diagram (\ref{siteandpresentables}). Analogously, we see that the right triangle of the diagram is commutative as well. This observation also shows that the lower square is also commutative by definition. All of these, together with the fact that $\mathparagraph_{\kappa} \, \mathrm{S}_{\kappa}(\BBB,\tau_{\BBB}) \hookrightarrow  \Sh(\BBB,\tau_{\BBB})$ is fully faithful, provides us that there exists a canonical isomorphism 
\begin{equation}
\eta(f): \eta_{(\BBB,\tau_{\BBB})} \, f \overset{\sim}{\Rightarrow} \mathparagraph_{\kappa}\, \mathrm{S}_{\kappa}(f) \, \eta_{(\AAA,\tau_{\AAA})}
\end{equation}
such that the upper square within the diagram is commutative up to $\eta(f)$, that is, we have
\[
\begin{tikzcd}
(\AAA,\tau_{\AAA}) \ar[r, "f"] \ar[d,"\eta_{(\AAA,\tau_{\AAA})}"'] &(\BBB,\tau_{\BBB})  \ar[d,"\eta_{(\BBB,\tau_{\BBB})}"] \ar[dl,shorten <=4.5ex,shorten >=4.5ex, Rightarrow,"\eta(f)"']\\
\mathparagraph_{\kappa} \, \mathrm{S}_{\kappa} (\AAA, \tau_{\AAA})  \ar[r,"\mathparagraph_{\kappa}\, \mathrm{S}_{\kappa}(f)"'] &\mathparagraph_{\kappa} \, \mathrm{S}_{\kappa}(\BBB,\tau_{\BBB}).
\end{tikzcd}
\]
From the fact that $\Phi_f$ is a canonical isomorphism, determined by the adjunction $\mathrm{S}_{\kappa}(f)=f^s \dashv f_s$ (as well as the adjunctions $L_{\AAA} \dashv i_{\AAA}$ and $L_{\BBB} \dashv i_{\BBB}$, see \cite[Expos\'e iii, Prop 1.2 \& Prop 1.3]{SGA4-1}), which is natural on $f$, one concludes that $\Phi_f$ is natural on $f$. Consequently, so is $\eta(f)$ and hence we have that $\eta$ is compatible with compositions and units. In addition, from the naturality of $\eta$ and the definitions, one can easily see that
\[
\begin{tikzcd}[column sep=7ex,row sep=8ex]
(\AAA,\tau_{\AAA}) \ar[r, "f"] \ar[d,"\eta_{(\AAA,\tau_{\AAA})}"'] &(\BBB,\tau_{\BBB})  \ar[d,"\eta_{(\BBB,\tau_{\BBB})}"] \ar[dl,shorten <=7ex,shorten >=7ex, shift right=2, Rightarrow,"\eta(f)"'] &[-6ex] \, \ar[d,phantom,"=" description]&[-6ex] (\AAA,\tau_{\AAA}) \ar[r,Rightarrow,start anchor={[yshift=1ex,xshift=6.25ex]},end anchor={[yshift=-1ex,xshift=-6.25ex]},shift right=1,"\alpha"] \ar[r, bend left=15, "f"] \ar[r, bend right=15, "g"'] \ar[d,"\eta_{(\AAA,\tau_{\AAA})}"'] &(\BBB,\tau_{\BBB})  \ar[d,"\eta_{(\BBB,\tau_{\BBB})}"] \ar[d,"\eta_{(\BBB,\tau_{\BBB})}"] \ar[d,"\eta_{(\BBB,\tau_{\BBB})}"] \ar[dl,shorten <=7ex,shorten >=7ex, Rightarrow,"\eta(g)"]\\
\mathparagraph_{\kappa} \, \Sh (\AAA, \tau_{\AAA}) \ar[r, bend left=15, "\mathparagraph_{\kappa}\, \mathrm{S}_{\kappa}(f)"] \ar[r, bend right=15, "\mathparagraph_{\kappa}\, \mathrm{S}_{\kappa}(g)"'] \ar[r,Rightarrow,start anchor={[yshift=1ex,xshift=3.45ex]},end anchor={[yshift=-1ex,xshift=-3.45ex]},shift right=5,"\mathparagraph_{\kappa}\, \mathrm{S}_{\kappa}(\alpha)"]&\mathparagraph_{\kappa} \, \Sh(\BBB,\tau_{\BBB}) &[-6ex] \, &[-6ex] \mathparagraph_{\kappa} \, \Sh (\AAA, \tau_{\AAA})  \ar[r,"\mathparagraph_{\kappa}\, \mathrm{S}_{\kappa}(g)"'] &\mathparagraph_{\kappa} \, \Sh(\BBB,\tau_{\BBB}).
\end{tikzcd}
\]
A detailed check of this naturality condition can be found in \cite[Fig (22)]{ramosgonzalez18}. Hence, we have that $\eta: \Id_{\Sites_{\kappa}} \Rightarrow \mathparagraph_{\kappa} \, \mathrm{S}_{\kappa}$ is a pseudonatural transformation, providing us the unit of the biadjunction. 

It only remains to check that the unit and the counit satisfy the triangle identities up to invertible modifications. Actually, in this case something stronger is true, namely the triangle identities hold strictly. Indeed, for every $(\CCC,\tau_{\CCC})$ in $\Sites_{\kappa}$, we have that
$$\epsilon_{\Sh(\CCC,\tau_{\CCC})} \, \mathrm{S}_{\kappa}(\eta_{(\CCC,\tau_{\CCC})}) = \epsilon_{\Sh(\CCC,\tau_{\CCC})} \, (\epsilon_{\Sh(\CCC,\tau_{\CCC})})^{-1} = \Id_{\Sh(\CCC,\tau_{\CCC})},$$
where the second equality follows directly from (\ref{siteandpresentables}). This shows that
\begin{equation*}
(\epsilon \, \mathrm{S}_{\kappa}) \, (\mathrm{S}_{\kappa} \, \eta) = \Id_{\mathrm{S}_{\kappa}}.
\end{equation*} 
On the other hand, for every $\mathcal{G} \in \Grtopoi_{\kappa}$, the following diagram
\begin{equation*}
	\begin{tikzcd}
	&\mathparagraph_{\kappa}(\mathcal{G}) \ar[dl,hook] \arrow{d}{\alpha_{\mathparagraph_{\kappa}(\mathcal{G})}} \arrow{r}{\eta_{\mathparagraph_{\kappa}(\mathcal{G})}} &\mathparagraph_{\kappa} \, \mathrm{S}_{\kappa} \, \mathparagraph_{\kappa} (\mathcal{G}) \ar[d,hook] \arrow{r}{\mathparagraph_{\kappa}(\epsilon_{\mathcal{G}})}& \mathparagraph_{\kappa}(\mathcal{G}) \ar[d,hook]\\
	\mathcal{G} \arrow{r}[below]{(\epsilon_{\mathcal{G}})^{-1}}&\mathrm{S}_{\kappa} \, \mathparagraph_{\kappa}(\mathcal{G}) \ar[r,equal]& \mathrm{S}_{\kappa} \, \mathparagraph_{\kappa}(\mathcal{G}) \arrow{r}[below]{\epsilon_{\mathcal{G}}} &\mathcal{G}
	\end{tikzcd}
\end{equation*}
is commutative. Observe that the horizontal composition below is the identity of on $\mathcal{G}$, and thus we have that $\mathparagraph_{\kappa}(\epsilon_{\mathcal{G}}) \, \eta_{\mathparagraph_{\kappa}(\mathcal{G})}$ is the restriction of the identity on $\cg$ to $\kappa$-presentable objects, that is, we have that
$$\mathparagraph_{\kappa}(\epsilon_{\mathcal{G}}) \, \eta_{\mathparagraph_{\kappa}(\mathcal{G})} = \Id_{\mathparagraph_{\kappa}(\mathcal{G})}.$$
This proves that the second triangle identity
\begin{equation*}
(\mathparagraph_{\kappa} \, \epsilon) \, (\eta \, \mathparagraph_{\kappa}) = \Id_{\mathparagraph_{\kappa}}
\end{equation*} 
also holds. Hence $\mathrm{S}_{\kappa}: \Sites_{\kappa} \rightleftarrows \Grtopoi_{\kappa}: \mathparagraph_{\kappa}$ is a biadjunction, as we wanted to show.
 
Now, observe that as the counit of the biadjunction is a pseudonatural equivalence, we have that $\mathparagraph_{\kappa}$ is a bi-fully-faithful pseudofunctor, i.e. it presents $\Grtopoi_{\kappa}$ as a full sub-bicategory of $\Sites_{\kappa}$. 
\end{proof}

\begin{rem} \label{moment}
	The reader should notice that \Cref{semantic} does not lead to the definition of a bireflection of $\Grtopoi$ in $\Sites^{\coop}$ (that in any case we already know cannot exist, as proved in \Cref{reflection}). Observe that, although 
	\begin{itemize}
		\item the family $(\Grtopoi_{\kappa})_{\kappa \in \RegCard}$ is cofinal in $\Grtopoi$ (see \S\ref{cofinality}),
		\item the family $(\text{Site}_{\kappa})_{\kappa \in \RegCard}$ is cofinal in $\Sites^{\coop}$ (for each site $(\CCC,\tau_{\CCC})$ it is enough to take a regular cardinal $\kappa$ strictly bigger than the cardinality of arrows of $\CCC$) and
		\item for every $\alpha \leq \beta$ regular cardinals the diagram
		\[
		\begin{tikzcd}
		\Sites_{\alpha}^{\coop} \arrow[d, hook] \arrow[r, "S_{\alpha}"] & \Grtopoi_{\alpha} \arrow[d, hook] \\
		\Sites_{\beta}^{\coop} \arrow[r, "S_{\beta}"] & \Grtopoi_{\beta}
		\end{tikzcd}
		\] 
		is commutative,
	\end{itemize}
	  we have that, given $\alpha \leq \beta$ regular cardinals, the diagram
	  \[
	  \begin{tikzcd}
	  \Grtopoi_{\alpha} \arrow[d, hook] \arrow[r, "\mathparagraph_{\alpha}"] & \Sites_{\alpha}^{\coop} \arrow[d, hook] \\
	  \Grtopoi_{\beta} \arrow[r, "\mathparagraph_{\beta}"] & \Sites_{\beta}^{\coop}
	  \end{tikzcd}
	  \] 
	  is not commutative, which avoids the bi-reflection for each cardinal to extend to a bi-reflection of $\Grtopoi$ in $\Sites^{\coop}$.
\end{rem}

As a corollary of the previous theorem, together with \Cref{1} one gets a connection between the site presentation and the syntactic presentation of a topos.
\begin{cor}\label{syntactic}
The $2$-category $\Prototopoi_{\kappa}$ is bi-reflective in $\Sites_{\kappa}$ via the following biadjunction,
$$\wp_{\kappa}:  \Sites_{k} \rightleftarrows  \Prototopoi_{\kappa}:  \mathrm{T}_{\kappa}, $$
where $\wp_{\kappa}$ is given by the composition $\Pres_{\kappa} \, \mathrm{S}_{\kappa}$.
\end{cor}

\begin{rem}
The unit of this biadjunction $(\CCC,\tau_{\CCC}) \to  \mathrm{T}_{\kappa} \, \wp_{\kappa}(\CCC,\tau_{\CCC})$ establishes an envelope of the site $(\CCC,\tau_{\CCC})$ amongst $\kappa$-limit theories. This completion is described in \Cref{kappasmall} and coincides with the closure of $\CCC$ under $\kappa$-small colimits in $ \Sh (\CCC,\tau_{\CCC})$. 
\end{rem}

\begin{rem}
We shall finish with an observation which is very much in the spirit of \cite{caramello18} and \cite{ramosgonzalez18}. Putting together  \Cref{1} and \Cref{syntactic}, two sites in $\Sites_{\kappa}$ have the same category of sheaves if and only if $\mathrm{T}_{\kappa} \, \wp_{\kappa}$ maps them in two equivalent protopoi. Writing it down explicitly, one gets the following theorem.
\end{rem}

\begin{thm} Let $f: (\AAA,\tau_{\AAA}) \to (\BBB,\tau_{\BBB})$ be a $1$-cell in $\Sites_{\kappa}$. Denote by $\AAA_\kappa$ (resp. $\BBB_\kappa$) the closure of $(\AAA,\tau_{\AAA})$ (resp. $(\BBB,\tau_{\BBB})$) under $\kappa$-small colimits in $\Sh(\AAA,\tau_{\AAA})$ (resp. in $\Sh(\BBB,\tau_{\BBB})$) as in \Cref{kappasmall}. Then, the following are equivalent:
\begin{enumerate}[label=\rm (\arabic*)]
\item $\mathrm{S}_{\kappa} (f)$ induces an equivalence of topoi.
\item The restriction of $\mathrm{S}_{\kappa} (f)$ to $\AAA_\kappa$ and $\BBB_\kappa$ is an equivalence of categories.
\end{enumerate}
\end{thm}

\begin{rem}
One can see this last theorem as a version of the generalization of the Lemme de comparaison from \cite[Cor. 4.5]{lowen04} in the setup of weakly $\kappa$-ary sites. 
\end{rem}

\begin{rem}
It is well known that two categories have the same presheaf category if and only if their Cauchy completion (also known as Karoubi completion or idempotent completion) are equivalent (see, for example \cite[Thm. 3.6']{elkinszilber76}). Roughly speaking, in a general Morita theoretical framework, two categories $\AAA$ and $\BBB$ are Morita-equivalent (meaning that their categories of models of a certain nature are equivalent) if and only if they are equivalent after taking a certain completion of them. In our context this completion is precisely the one induced by the monad $\mathrm{T}_{\kappa} \, \wp_{\kappa}$.
\end{rem}

\section*{Acknowledgements}
The authors are very grateful to the organisers of the summer school and conference \textit{Toposes in Como} (Università degli Studi dell’Insubria - Como, June 2018) during which the initial discussions that led to this paper took place. The first author is also indebted with Ji\v{r}\'i Rosick\'{y} for pointing out some very useful references. The second author would like to thank Jens Hemelaer for very interesting discussions on topos theory and in particular for his help in fully understanding the embedding $\Ind_{\kappa}(\CCC) \rightarrow \CCC^{\op}$-$\Alg$, which appears in the proof of \Cref{extensive}. She is also grateful to Wendy Lowen for the reading of this manuscript and for her useful suggestions, as well as for pointing out reference \cite{breitsprecher70}. Finally, she would like to thank the organisers of the \textit{S\'eminaire de th\'eorie des cat\'egories} (Universit\'e catholique de Louvain, Universit\'e libre de Bruxelles and Vrije Universiteit Brussel) for giving her the opportunity to present this work during the seminar, which resulted in an improvement of the presentation of the paper. 


\bibliography{bib}
\bibliographystyle{abbrv}

\end{document}